\documentclass[reqno]{amsart}

\usepackage[T1]{fontenc}
\usepackage{amsmath,amssymb,mathtools}
\usepackage{enumitem}
\usepackage{graphicx}
\usepackage{xcolor}
\usepackage{float}
\usepackage{tikz}
\usetikzlibrary{arrows.meta,calc,fit,positioning}
\usepackage{pgfplots}
\pgfplotsset{compat=1.18}
\usepackage{hyperref}
\usepackage{microtype}

\definecolor{OrbitBlue}{HTML}{155C8A}
\definecolor{OrbitTeal}{HTML}{168A83}
\definecolor{OrbitGold}{HTML}{D49A28}
\definecolor{OrbitRed}{HTML}{B44C43}
\definecolor{OrbitInk}{HTML}{263746}
\definecolor{OrbitMist}{HTML}{EEF4F7}

\numberwithin{equation}{section}
\allowdisplaybreaks
\newtheorem{theorem}{Theorem}[section]
\newtheorem{proposition}[theorem]{Proposition}
\newtheorem{lemma}[theorem]{Lemma}
\newtheorem{corollary}[theorem]{Corollary}
\theoremstyle{definition}
\newtheorem{definition}[theorem]{Definition}
\newtheorem{example}[theorem]{Example}
\theoremstyle{remark}
\newtheorem{remark}[theorem]{Remark}

\newcommand{\cH}{\mathcal H}
\newcommand{\cI}{\mathcal I}
\newcommand{\N}{\mathbb N}
\newcommand{\R}{\mathbb R}

\newcommand{\T}{\mathbb T}
\newcommand{\U}{\mathcal U}
\newcommand{\PU}{\operatorname{PU}}
\newcommand{\pker}{\operatorname{pker}}
\newcommand{\Ad}{\operatorname{Ad}}
\newcommand{\ad}{\operatorname{ad}}
\newcommand{\HS}{\mathrm{HS}}
\newcommand{\e}{\mathrm e}
\newcommand{\Step}[1]{\par\medskip\noindent\textit{#1.}\quad}
\newcommand{\LeanSourceLink}[1]{{\small\nolinkurl{#1}}}

\title[Near-Parseval orbits from mixing and expansion]
{Near-Parseval orbit frames for irreducible unitary representations:\\
from mixing and expansion}

\author{Vignon Oussa}
\address{Department of Mathematics, Bridgewater State University,
Bridgewater, MA 02325}
\email{voussa@bridgew.edu}

\hypersetup{
  colorlinks=true,
  linkcolor=OrbitBlue,
  citecolor=OrbitTeal,
  urlcolor=OrbitBlue,
  pdftitle={Near-Parseval orbit frames for irreducible unitary representations: from mixing and expansion},
  pdfauthor={Vignon Oussa},
  pdfkeywords={orbit frame, irreducible unitary representation, exponential solvable Lie group, projective kernel, Howe-Moore decay}
}

\begin{document}

\begin{abstract}
We establish a sequential storage theorem for unitary representations of
second-countable locally compact groups.  More precisely, we assume
weak-operator decay along one storage sequence together with weak-operator
decay of the conjugates of every nontrivial element in a local retrieval
family.  Under these assumptions, and for every $0<\varepsilon<1$, we
construct a single discrete orbit with frame bounds $(1-\varepsilon)^2$ and
$(1+\varepsilon)^2$, whose sampling set is both quotient-injective and
relatively separated modulo the projective kernel.  Although projective
$C_0$ decay and geometric conjugation escape provide a natural sufficient
condition, neither solvability nor Lie structure enters the abstract theorem.
Consequently, the same mechanism applies in several different settings.  It
shows that every infinite-dimensional irreducible unitary representation of
an exponential solvable Lie group admits a near-Parseval orbit frame, and it
gives the same conclusion for every infinite-dimensional irreducible unitary
representation of a connected semisimple real Lie group with finite center.
In the latter case, the required hypotheses follow by combining Howe--Moore
decay on an active noncompact simple factor with expansion in a restricted-root
subgroup.  In addition, the sequential formulation applies to
$\operatorname{SL}_d(\mathbb k)$ over every local field $\mathbb k$, including
the totally disconnected non-Archimedean cases.  Finally, when the effective
projective quotient in the exponential solvable case is abelian, the
near-Parseval conclusion can be strengthened: a transported Weyl lattice
produces an orbit which is an orthonormal basis.
\end{abstract}

\maketitle
\enlargethispage{4pt}

\section{Introduction}\label{sec:introduction}

\subsection{Orbit frames and projective kernels}

Let $G$ be a second-countable locally compact group and let
$\pi:G\to\U(\cH)$ be a strongly continuous unitary representation on a
separable Hilbert space.  We are concerned with the following orbit-frame
problem: determine whether one can find a single nonzero vector
$\phi\in\cH$ and a countable set $\Gamma\subset G$ such that
\[
   \{\pi(\gamma)\phi:\gamma\in\Gamma\}
\]
is a frame for $\cH$.  Equivalently, one seeks constants
$0<A\leq B<\infty$ satisfying
\begin{equation}\label{eq:intro-frame-definition}
 A\|f\|^2
 \leq \sum_{\gamma\in\Gamma}
       |\langle f,\pi(\gamma)\phi\rangle|^2
 \leq B\|f\|^2
 \qquad(f\in\cH).
\end{equation}
The essential feature of the problem is that every frame vector must belong
to the same group orbit.  On the other hand, no subgroup or lattice condition
is imposed on $\Gamma$, and this flexibility will be important in the
construction below.

For a unitary representation $\pi$, its \emph{projective kernel} is
\begin{equation}\label{eq:projective-kernel}
 K=\pker(\pi)
   :=\{g\in G:\pi(g)\in\T I_{\cH}\}.
\end{equation}
Since elements of $K$ change an orbit vector only by a unimodular scalar, the
quotient $Q=G/K$ is the effective parameter space for the orbit.  It is
therefore natural to formulate the geometry of the sampling set in $Q$.  A
subset $\Lambda$ of a locally compact group $Q$ is called \emph{left
relatively separated} if, for some relatively
compact identity neighborhood $V\subset Q$,
\begin{equation}\label{eq:relative-separation}
  \sup_{x\in Q}\#(\Lambda\cap xV)<\infty.
\end{equation}
The preceding condition is independent of the particular relatively compact
identity neighborhood.  Moreover, every relatively separated subset is
locally finite and is therefore closed and discrete.

In the principal solvable setting of this paper, every irreducible
representation is monomial \cite{Pukanszky1968,LeptinLudwig1994}.  This
structural fact is useful, but monomiality alone does not resolve the
discretization problem.  Indeed, a polarization need not be normal, its
normal core may provide too few scalar modulation parameters, and fiberwise
frames need not share a single sampling set with uniform bounds.  For this
reason, our construction is carried out directly in the representation
space, where these compatibility questions do not arise.

Before placing the result in its historical setting, we fix two conventions.
Throughout the paper, inner products are linear in the first variable and
$\N=\{1,2,\ldots\}$.  We also recall that a solvable Lie group is called
\emph{exponential} when its exponential map is a global diffeomorphism; in
particular, such a group is connected and simply connected.

\subsection{Historical position and scope of the novelty}
\label{subsec:history-novelty}

The frame inequality entered harmonic analysis through the nonharmonic
Fourier series of Duffin and Schaeffer \cite{DuffinSchaeffer1952}.
Subsequently, the coherent-state viewpoint organized overcomplete systems as
orbits of a group representation \cite{Perelomov1972}, and Gabor and affine
systems became central examples \cite{DaubechiesGrossmannMeyer1986}.  In
particular, for a square-integrable
irreducible representation, the Duflo--Moore orthogonality relations show that
an admissible vector generates, after normalization, a continuous tight frame
(or a continuous tight frame modulo the subgroup on which the representation
acts projectively) \cite{DufloMoore1976}.  Thus square-integrability supplies
a continuous reproducing formula before the question of discrete sampling is
addressed.

Coadjoint-orbit theory provides a different, but complementary, point of
view.  For exponential solvable
groups, the Kirillov--Puk\'anszky picture and induction from polarizations
organize the irreducible dual in terms of coadjoint geometry
\cite{Pukanszky1968,ArnalCurrey2020}.  This supplies the structural origin of
the representations and of their projective kernels, but it does not by
itself select one vector and a discrete set whose orbit satisfies frame
bounds.  The present construction uses those representation-theoretic inputs
only to obtain coefficient decay and conjugation escape.  Once these two
inputs have been established, the frame itself is constructed by a separate
storage argument.

The relationship with coorbit theory deserves a more precise explanation.
Classical coorbit theory, initiated by Feichtinger and Gr\"ochenig, starts
from an integrable irreducible unitary representation and a sufficiently
localized analyzing vector $g$.  In terms of the voice transform
\[
   V_gf(x)=\langle f,\pi(x)g\rangle,
\]
the hypothesis is expressed through weighted integrability of the reproducing
kernel $V_gg$.  The theory constructs Banach spaces by imposing function-space
norms on $V_gf$ and, for sufficiently dense well-spread sampling sets,
produces atomic decompositions and Banach frames simultaneously across the
associated scale of coorbit spaces
\cite{FeichtingerGrochenig1988,FeichtingerGrochenig1989I,
FeichtingerGrochenig1989II}.  Oscillation estimates later gave direct sampling
theorems for square-integrable representations and discrete series
\cite{FuehrGrochenig2007}.  These results provide substantially more
localization and coefficient-space information than is sought here, but their
starting assumptions do not hold for an arbitrary irreducible representation.
In particular, the mere vanishing of matrix coefficients at infinity does
not imply the square- or absolute-integrability required by those
constructions.  Thus the hypotheses used here and those used in classical
coorbit theory address genuinely different representation-theoretic regimes.

There is also a complementary line of work which begins with an arbitrary
continuous frame rather than with a group representation.  Fornasier and
Rauhut obtained discretization
under localization hypotheses by extending coorbit methods
\cite{FornasierRauhut2005}.  Freeman and Speegle subsequently proved, in
particular, that every bounded continuous frame can be sampled to obtain a
discrete frame \cite{FreemanSpeegle2019}.  A unitary orbit has constant vector
norm, so it is bounded; nevertheless, this theorem becomes applicable to an
orbit only after one knows that the orbit is a continuous frame.  It therefore
does not manufacture an admissible vector for a representation that is not
square-integrable.  Moreover, its abstract sampling theorem allows repeated
sample points and does not assert quotient injectivity, relative separation
in a group quotient, or arbitrarily near-Parseval bounds.

In addition to these general discretization theorems, there are important
representation-specific constructions which do not pass through classical
coorbit theory.  Under a semidirect-product hypothesis
$G=P\rtimes M$ with a normal inducing subgroup $P$, explicit tight frames for
certain irreducible representations of completely solvable groups were
constructed in \cite{Oussa2018}.  The later work \cite{Oussa2019} treats
representations induced from characters of a normal exponential subgroup
under a coadjoint-immersion hypothesis; the representation need not be
irreducible or integrable, and the construction yields compactly supported
bounded, continuous or smooth windows, with Parseval conclusions in the
exponential case.  For nilpotent groups, Gr\"ochenig and Rottensteiner
constructed orbit orthonormal bases for square-integrable representations
modulo the center under broad structural hypotheses
\cite{GrochenigRottensteiner2018}; Oussa subsequently proved the
orthonormal-basis conclusion for every generic irreducible representation of
a connected, simply connected nilpotent Lie group, hence in particular for
every representation square-integrable modulo the center \cite{Oussa2024}.
On the classes treated in those papers, the resulting exact tight-frame and
orthonormal-basis conclusions are stronger than a near-Parseval conclusion.
The irreducible orbit-frame problem in the wider generality considered here
was isolated explicitly in \cite{Oussa2026}; the question is therefore how
far one can proceed after the assumptions that yield those exact
constructions are removed.

Against this background, the present paper begins at a different point.
Theorem~\ref{thm:abstract-storage} does not discretize a pre-existing
continuous frame and assumes neither an admissible vector nor an integrable
reproducing kernel.  It starts from two directional weak-operator limits:
decay of a storage sequence and decay of the conjugates of nontrivial elements
near the identity.  The proof encodes a redundant Parseval target into one
square-summable storage vector and uses conjugation to retrieve its columns
with a prescribed synthesis-operator error.  This produces, for every
$\varepsilon>0$, a one-vector orbit frame with bounds
$(1-\varepsilon)^2$ and $(1+\varepsilon)^2$, while simultaneously making the
sampling set injective and relatively separated in the quotient by the
projective kernel.  Consequently, the new abstract mechanism is the
storage-and-retrieval step itself, rather than the sampling of a continuous
reproducing formula.

We first apply this mechanism to exponential solvable Lie groups, where it
combines projective $C_0$ decay with conjugation escape.  Relative to the
preceding results, the
new scope of Theorem~\ref{thm:main} is its uniform quantification over
\emph{every} infinite-dimensional irreducible unitary representation of
\emph{every} exponential solvable Lie group, without a square-integrability,
normal-polarization, semidirect-product, rationality, or prescribed
induced-model hypothesis.  The price for this wider scope is equally
important: in the nonabelian branch the construction does not provide an
exact Parseval frame, a lattice or subgroup of sampling points, or a compactly
supported, smooth, or coorbit-localized generator.  When the effective
quotient is abelian, however, the situation is different.  In that case the
classical Weyl--Stone--von Neumann structure yields an orthonormal basis and
hence a stronger conclusion.

Once the abstract mechanism has been separated from its
representation-theoretic input, it becomes possible to pass beyond solvable
groups.
Howe--Moore vanishing is an established theorem
\cite{HoweMoore1979,Ciobotaru2014}, and root-subgroup expansion is standard
structure theory; neither ingredient is claimed as new.  Their combination
with Theorem~\ref{thm:abstract-storage} yields the near-Parseval orbit-frame
conclusion for every infinite-dimensional irreducible representation of a
connected semisimple real Lie group with finite center and for every such
representation of $\operatorname{SL}_d(\mathbb k)$ over an arbitrary local
field.  The cited coorbit and continuous-frame discretization results already
cover many square-integrable, holomorphic, or otherwise admissible
representations in these settings, but they do not by themselves yield this
conclusion for irreducible representations that are not square-integrable.
Accordingly, the novelty is not the existence of coherent states, coorbit
discretizations, or special orbit bases; it is the representation-uniform
orbit-frame conclusion obtained from directional mixing and storage,
together with its quantitative near-Parseval and projective sampling geometry.

\begin{remark}[Boundary of the novelty claim]
For each $\varepsilon$ the vector and sampling set may change; the result does
not assert that one fixed orbit becomes asymptotically Parseval as
$\varepsilon\downarrow0$.  Relative separation is asserted in
$G/\pker(\pi)$, not necessarily in $G$.
Theorem~\ref{thm:local-field-main} is stated for
$\operatorname{SL}_d(\mathbb k)$; the root-group argument suggests wider
split almost-simple classes, but those should be advertised as extensions
only when all hypotheses are stated and proved.  Finally, no priority claim
rests solely on the recent preprint \cite{Oussa2026}; the novelty statement
above compares the hypotheses and conclusions directly.
\end{remark}

\subsection{Main results}

We now state the principal consequences of the paper.  The first theorem
settles the orbit-frame problem, with quantitative bounds and controlled
sampling geometry, throughout the exponential solvable class.

\begin{theorem}[Irreducible orbit frames]\label{thm:main}
Let $G$ be an exponential solvable Lie group and let
$\pi:G\to\U(\cH)$ be an infinite-dimensional irreducible strongly continuous
unitary representation on a separable Hilbert space.  For every
$0<\varepsilon<1$,
there exist $0\neq\phi\in\cH$ and a countable set $\Gamma\subset G$ such
that
\begin{equation}\label{eq:near-parseval-frame}
 (1-\varepsilon)^2\|f\|^2
 \leq \sum_{\gamma\in\Gamma}
       |\langle f,\pi(\gamma)\phi\rangle|^2
 \leq (1+\varepsilon)^2\|f\|^2
 \qquad(f\in\cH).
\end{equation}
Moreover, if $q:G\to G/\pker(\pi)$ is the quotient map, then $\Gamma$
may be chosen so that $q$ is injective on $\Gamma$ and $q(\Gamma)$ is left
relatively separated.  If $G/\pker(\pi)$ is abelian, one may choose $\phi$ to be a
unit vector and the orbit to be an orthonormal basis; in that case
$q(\Gamma)$ is
a lattice in the effective quotient.
\end{theorem}

The argument leading to the first theorem is formulated sequentially and does
not depend on Lie theory.  Over a local field, Howe--Moore decay supplies the
weak-operator limit, while a root subgroup supplies the required expansion.
Consequently, we obtain the following result.

\begin{theorem}[Special linear groups over local fields]
\label{thm:local-field-main}
Let $\mathbb k$ be a local field, let $d\geq2$, and let
\[
 \pi:\operatorname{SL}_d(\mathbb k)\longrightarrow\U(\cH)
\]
be an infinite-dimensional irreducible strongly continuous unitary
representation on a separable Hilbert space.  For every
$0<\varepsilon<1$, there exist $0\neq\phi\in\cH$ and a countable set
$\Gamma\subset\operatorname{SL}_d(\mathbb k)$ satisfying
\eqref{eq:near-parseval-frame}.  If
$q:\operatorname{SL}_d(\mathbb k)\to
\operatorname{SL}_d(\mathbb k)/\pker(\pi)$ is the quotient map, then
$\Gamma$ may be chosen so that $q|_\Gamma$ is injective and $q(\Gamma)$ is
left relatively separated.
\end{theorem}

Over the real field, the same mechanism applies more generally to connected
semisimple groups.  We therefore obtain a result that lies genuinely outside
the solvable category.

\begin{theorem}[Semisimple real Lie groups]\label{thm:simple-main}
Let $G$ be a connected semisimple real Lie group with finite center,
and let $\pi:G\to\U(\cH)$ be an infinite-dimensional irreducible strongly
continuous unitary representation on a separable Hilbert space.  For every
$0<\varepsilon<1$, there exist $0\neq\phi\in\cH$ and a countable set
$\Gamma\subset G$ satisfying \eqref{eq:near-parseval-frame}.  If
$q:G\to G/\pker(\pi)$ is the quotient map, then $\Gamma$ may be chosen so
that $q|_\Gamma$ is injective and $q(\Gamma)$ is left relatively separated.
\end{theorem}

Although these three theorems concern different classes of groups, they have
the same analytic engine, namely Theorem~\ref{thm:abstract-storage}.  Its most
economical hypotheses are directional: one asks for a nondiscrete topological
group $A$,
a continuous homomorphism $p:A\to G$, and a sequence $(h_m)$ in $G$ such that
\begin{equation}\label{eq:intro-weak-escape}
 \pi(h_m)\xrightarrow{\mathrm{WOT}}0,
 \qquad
 \pi\bigl(h_m p(a)h_m^{-1}\bigr)
       \xrightarrow{\mathrm{WOT}}0
 \quad\text{for every nonidentity $a$ near $e_A$}.
\end{equation}
Neither solvability, Lie structure, nor irreducibility is used in the storage
theorem.  Moreover, projective $C_0$ decay together with escape in $G/K$
implies \eqref{eq:intro-weak-escape}.  To prove Theorem~\ref{thm:main}, published
projective-$C_0$ results and exponential Lie-algebra geometry supply these
inputs when the effective quotient is nonabelian, whereas an abelian
effective quotient is placed in Weyl normal form.  In the semisimple real
case, Howe--Moore decay and a restricted-root subgroup provide the same two
inputs directly.  Finally, over an arbitrary local field, a diagonal sequence
expands a single matrix root subgroup.  Figure~\ref{fig:scope-map} summarizes
these parallel routes to the storage theorem.

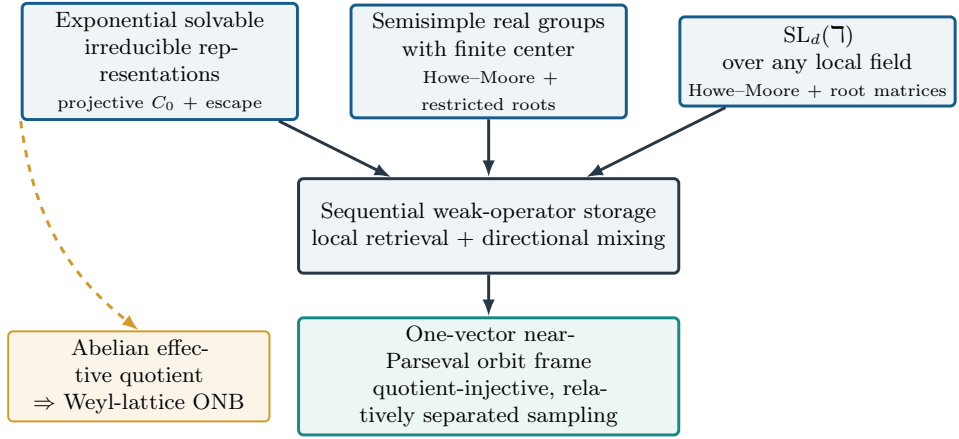
\begin{figure}[t]
\centering
\resizebox{\linewidth}{!}{%
\begin{tikzpicture}[
  font=\small,
  >={Latex[length=2.4mm]},
  source/.style={draw=OrbitBlue, very thick, rounded corners=2pt,
    fill=OrbitBlue!7, align=center, text width=3.7cm, minimum height=1.25cm},
  engine/.style={draw=OrbitInk, very thick, rounded corners=2pt,
    fill=OrbitMist, align=center, text width=5.2cm, minimum height=1.35cm},
  result/.style={draw=OrbitTeal, very thick, rounded corners=2pt,
    fill=OrbitTeal!8, align=center, text width=5.2cm, minimum height=1.25cm},
  side/.style={draw=OrbitGold, thick, rounded corners=2pt,
    fill=OrbitGold!10, align=center, text width=3.5cm, minimum height=1.1cm},
  link/.style={->, very thick, draw=OrbitInk}
]
\node[source] (solv) at (-4.7,2.35)
  {Exponential solvable\par irreducible representations\par
   {\scriptsize projective $C_0$ $+$ escape}};
\node[source] (real) at (0,2.35)
  {Semisimple real groups\par with finite center\par
   {\scriptsize Howe--Moore $+$ restricted roots}};
\node[source] (local) at (4.7,2.35)
  {$\operatorname{SL}_d(\mathbb k)$\par over any local field\par
   {\scriptsize Howe--Moore $+$ root matrices}};
\node[engine] (engine) at (0,0)
  {Sequential weak-operator storage\par
   local retrieval $+$ directional mixing};
\node[result] (frame) at (0,-2.15)
  {One-vector near-Parseval orbit frame\par
   quotient-injective, relatively separated sampling};
\node[side] (weyl) at (-5.0,-2.15)
  {Abelian effective quotient\par $\Rightarrow$ Weyl-lattice ONB};

\draw[link] (solv) -- (engine);
\draw[link] (real) -- (engine);
\draw[link] (local) -- (engine);
\draw[link] (engine) -- (frame);
\draw[->,very thick,dashed,draw=OrbitGold] (solv.south west)
  to[bend right=22] (weyl.north);
\end{tikzpicture}
}
\caption{Logical scope of the paper.  The central box is group-agnostic;
the upper boxes provide independent sources of its two directional WOT
limits.  Only the abelian effective quotient uses the separate Weyl branch.}
\label{fig:scope-map}
\end{figure}

\subsection{Architecture and organization}

We conclude the introduction by describing the organization of the proof.
There are five conceptual components, and each one supplies an input needed
by the next.
\begin{enumerate}[label=\textnormal{\textbf{Step \arabic*.}},leftmargin=*]
\item A synthesis-operator perturbation lemma converts operator-norm
closeness to a Parseval frame into quantitative frame bounds.  A separate
Bessel-orbit lemma upgrades the sampling geometry to relative separation.
\item The sequential storage theorem encodes a highly redundant Parseval
frame into one vector.  Conjugated weak-operator decay separates entries
within each block, ordinary weak-operator decay separates a new block from
all preceding blocks, and superdecreasing weights control every future block.
\item Noncommutativity and the spectral characterization of exponential Lie
algebras produce the required escape directions for nonabelian effective
quotients of exponential solvable groups.
\item Howe--Moore decay and root-subgroup expansion prove the nonsolvable
semisimple real and local-field applications.
\item An abelian effective quotient in the exponential solvable case is
reduced to the standard Weyl representation by multiplier cohomology and
Stone--von Neumann uniqueness.
\end{enumerate}

These five components also determine the organization of the paper.
Section~\ref{sec:preliminaries} develops the analytic and
representation-theoretic preliminaries.  Section~\ref{sec:storage} proves the
abstract storage theorem.  Section~\ref{sec:geometry} constructs conjugation
escape for exponential solvable groups.  Section~\ref{sec:beyond-exponential}
proves Theorems~\ref{thm:simple-main} and \ref{thm:local-field-main}.
Section~\ref{sec:abelian} treats the
abelian effective quotient, and Section~\ref{sec:proof-main} proves
Theorem~\ref{thm:main}.  Section~\ref{sec:B4-model} gives a complete
nonnilpotent $4\times4$ triangular model.  Appendix~\ref{app:lean} records the exact scope of
the machine-checked supplement and links to its Lean source files.

\section{Analytic and representation-theoretic preliminaries}
\label{sec:preliminaries}

\subsection{Synthesis operators and stable perturbations}

We begin with the elementary perturbation principle that will convert the
storage estimates into frame bounds.  Let $I$ be countable.  Every family
$(u_i)_{i\in I}$ in $\cH$ has an
algebraic synthesis map
\[
 D_{u,\mathrm{fin}}:c_{00}(I)\longrightarrow\cH,
 \qquad D_{u,\mathrm{fin}}c=\sum_{i\in I}c_i u_i.
\]
The family is \emph{Bessel} when this map extends to a bounded operator
$D_u:\ell^2(I)\to\cH$.  Its adjoint is then the analysis operator
$D_u^*f=(\langle f,u_i\rangle)_{i\in I}$.  The family is Parseval precisely
when $D_u^*$ is an isometry.  Thus it is enough to control the synthesis
operator of the orbit family in norm, and the following lemma gives the
required quantitative statement.

\begin{lemma}[Synthesis perturbation]
\label{lem:synthesis-perturbation}
Let $(u_i)_{i\in I}$ be a Parseval frame with synthesis operator $D_0$, and
let $(v_i)_{i\in I}$ be a Bessel family with synthesis operator $D$.  If
\begin{equation}\label{eq:synthesis-perturbation-hypothesis}
   \|D-D_0\|\leq E<1,
\end{equation}
then $(v_i)_{i\in I}$ is a frame with bounds $(1-E)^2$ and $(1+E)^2$.
\end{lemma}

\begin{proof}
Since $(u_i)$ is Parseval, $\|D_0^*f\|=\|f\|$.  We may therefore compare the
two analysis operators directly.  The triangle inequality and its reverse
give
\[
 \|D^*f\|
 \leq \|D_0^*f\|+\|(D-D_0)^*f\|
 \leq (1+E)\|f\|
\]
and
\[
 \|D^*f\|
 \geq \|D_0^*f\|-\|(D-D_0)^*f\|
 \geq (1-E)\|f\|.
\]
After squaring these two estimates, we obtain the claimed frame inequalities.
\end{proof}

\begin{example}[A diagonal perturbation]\label{ex:diagonal-perturbation}
Let $(e_n)$ be an orthonormal basis and let
$v_n=\lambda_ne_n$, where $|\lambda_n-1|\leq E<1$.  Relative to the standard
basis of $\ell^2$, the difference of the two synthesis operators is the
diagonal operator with entries $\lambda_n-1$.  Hence its norm is at most
$E$, and Lemma~\ref{lem:synthesis-perturbation} gives frame bounds
$(1-E)^2$ and $(1+E)^2$.  The storage construction below is an infinite,
nondiagonal version of this elementary comparison.
\end{example}

\subsection{Decay modulo the projective kernel}

We next turn to the representation-theoretic input.  Since scalar operators
do not affect coefficient magnitudes, all decay statements must be
formulated on the quotient by the projective kernel.  For this purpose,
equip $\U(\cH)$ with the strong operator topology and
\[
 \PU(\cH)=\U(\cH)/(\T I_{\cH})
\]
with the quotient topology.  A homomorphism into $\PU(\cH)$ is
\emph{projectively irreducible} when its unitary representatives have no
nontrivial common closed invariant subspace.

Fix a finite-dimensional Lie group $G$ and an irreducible strongly continuous
unitary representation $\pi:G\to\U(\cH)$.  Set $K=\pker(\pi)$ and
$Q=G/K$.  The following elementary dictionary distinguishes the three
quotient and lifting operations that will occur later and thereby fixes the
notation used throughout the remainder of the paper.

\begin{proposition}[The quotient dictionary]
\label{prop:quotient-dictionary}
Let $q:G\to Q$ be the quotient map.
\begin{enumerate}[label=\textnormal{(\roman*)}]
\item $K$ is a closed normal subgroup of $G$.
\item For $u,v\in\cH$, the function
\[
 gK\longmapsto |\langle u,\pi(g)v\rangle|
\]
is well defined and continuous on $Q$.
\item The formula
\[
 \overline\pi(gK)=[\pi(g)]
\]
defines a faithful continuous homomorphism
$\overline\pi:Q\to\PU(\cH)$.  It is projectively irreducible whenever
$\pi$ is irreducible.
\item The differential $dq_e:\mathfrak g\to\mathfrak q$ is surjective.  If
$dq_e(X)=\overline X$, then
\begin{equation}\label{eq:quotient-exp-intertwining}
 q\bigl(\exp_G(tX)\bigr)=\exp_Q(t\overline X)
 \qquad(t\in\R).
\end{equation}
Thus every one-parameter subgroup of $Q$ has a one-parameter lift to $G$.
\end{enumerate}
\end{proposition}

\begin{proof}
We verify the assertions in the order in which they are stated.  First, the
scalar subgroup $\T I_{\cH}$ is closed in $\U(\cH)$ with the strong
operator topology, so $K=\pi^{-1}(\T I_{\cH})$ is closed.  If $k\in K$,
then $\pi(k)$ is scalar and consequently
$\pi(gkg^{-1})=\pi(k)$ for every $g\in G$; hence $K$ is normal.

Write $\pi(k)=\zeta(k)I_{\cH}$ for $k\in K$.  Since $|\zeta(k)|=1$,
\[
 |\langle u,\pi(gk)v\rangle|
 =|\langle u,\pi(g)v\rangle|.
\]
Thus the coefficient magnitude is constant on cosets, and its
continuity on $Q$ follows from the quotient topology.

Similarly, the projective class $[\pi(g)]$ is constant on cosets.  Its kernel
in $Q$ is trivial by the definition of $K$, which proves faithfulness.
Continuity again follows from the quotient topology.  A closed subspace is
invariant under $\overline\pi(Q)$ precisely when it is invariant under one,
and hence every, unitary representative of each projective class.  Such
subspaces are exactly the $\pi(G)$-invariant subspaces, proving the
irreducibility assertion.

It remains to establish the lifting assertion.  Since $q$ is a Lie-group
homomorphism, $dq_e$ is the quotient map of Lie algebras.  Lie-group
homomorphisms intertwine exponential maps, which gives
\eqref{eq:quotient-exp-intertwining} and completes the proof.
\end{proof}

\begin{remark}[Three different lifts]\label{rem:three-lifts}
Part~(iv) lifts a Lie-algebra vector and its one-parameter subgroup; it does
not provide a global section $Q\to G$.  Choosing one representative in $G$
for each point of a discrete subset of $Q$ is only a set-theoretic choice.
Finally, lifting a map $Q\to\PU(\cH)$ to a continuous map with values in
$\U(\cH)$ is a separate principal-bundle question, used only in the abelian
branch.
\end{remark}

The next lemma explains why relative separation need not be imposed as an
additional hypothesis once an orbit family is known to be Bessel.  More
precisely, continuity of a single diagonal coefficient gives a uniform local
counting estimate in the effective quotient.

\begin{lemma}[Bessel orbit families are relatively separated]
\label{lem:bessel-relative-separation}
Let $G$ be a second-countable locally compact group, let
$\pi:G\to\U(\cH)$ be strongly continuous, put $K=\pker(\pi)$, and let
$q:G\to Q=G/K$ be the quotient map.  Suppose that $0\neq\phi\in\cH$ and
that $\{\pi(\gamma)\phi:\gamma\in\Gamma\}$ is a Bessel family with
Bessel bound $B\geq0$.  Then $q(\Gamma)$ is left relatively separated.  More precisely,
there is a relatively compact identity neighborhood $V\subset Q$ such that
\begin{equation}\label{eq:relative-separation-bound}
  \sup_{x\in Q}\#\bigl(q(\Gamma)\cap xV\bigr)
  \leq \frac{4B}{\|\phi\|^2}.
\end{equation}
In particular, $q(\Gamma)$ is locally finite and therefore closed and
discrete.
\end{lemma}

\begin{proof}
We first obtain a uniform coefficient estimate near the identity.  The
coefficient magnitude
\[
  q(g)\longmapsto |\langle\phi,\pi(g)\phi\rangle|
\]
is well defined and continuous on $Q$, and its value at the identity coset
is $\|\phi\|^2$.  Choose a relatively compact identity neighborhood
$V\subset Q$ such that
\begin{equation}\label{eq:coefficient-near-identity}
 |\langle\phi,\pi(g)\phi\rangle|
 \geq \frac12\|\phi\|^2
 \qquad\text{whenever }q(g)\in V.
\end{equation}
Having fixed $V$, let $x\in Q$ and choose $g_x\in G$ with $q(g_x)=x$.  If
$q(\gamma)\in xV$, then $q(g_x^{-1}\gamma)\in V$, and unitarity gives
\[
 |\langle\pi(g_x)\phi,\pi(\gamma)\phi\rangle|
 =|\langle\phi,\pi(g_x^{-1}\gamma)\phi\rangle|
 \geq \frac12\|\phi\|^2.
\]
We now apply the Bessel inequality to a finite portion of this cluster.  Let
$F\subset q(\Gamma)\cap xV$ be finite, and choose one
$\gamma_y\in\Gamma$ with $q(\gamma_y)=y$ for each $y\in F$.  Testing the
Bessel inequality against $\pi(g_x)\phi$ and retaining only these distinct
indices yields
\[
 \#F\,\frac{\|\phi\|^4}{4}
 \leq B\|\phi\|^2,
\]
so every finite subset has cardinality at most $4B/\|\phi\|^2$.  Consequently,
$q(\Gamma)\cap xV$ itself is finite with that bound, proving
\eqref{eq:relative-separation-bound}.  A compact subset of $Q$
is covered by finitely many sets of the form $xV$, so $q(\Gamma)$ meets every
compact set in finitely many points.  In a locally compact Hausdorff space,
such a subset is closed and discrete.  This proves both the quantitative
bound and the asserted sampling geometry.
\end{proof}

Having obtained the geometric consequence of the Bessel condition, we now
introduce the decay property that will be used to separate the stored blocks.

\begin{definition}\label{def:projective-c0}
A strongly continuous unitary representation $\pi$ is
\emph{projectively $C_0$} if
\begin{equation}\label{eq:projective-c0-definition}
 gK\longmapsto |\langle u,\pi(g)v\rangle|
 \quad\text{belongs to }C_0(G/K)
\end{equation}
for every $u,v\in\cH$.
\end{definition}

Here and below, $x_\alpha\to\infty$ in a locally compact space means that
the net eventually leaves every compact subset.

The next theorem supplies the only substantial representation-theoretic
input in the nonabelian exponential-solvable construction.  We state its
scope explicitly because, once this decay is available, the subsequent
storage argument is entirely abstract.

\begin{theorem}[Projective $C_0$ input]
\label{thm:projective-c0-input}
Let $G$ be an exponential solvable Lie group and let $\pi$ be an irreducible unitary
representation.  With $K=\pker(\pi)$, the following hold:
\begin{enumerate}[label=\textnormal{(\roman*)}]
\item $K$ is connected;
\item $G/K$ is an exponential solvable Lie group;
\item for every $w\in\cH$, the function
\[
 gK\longmapsto |\langle w,\pi(g)w\rangle|
\]
belongs to $C_0(G/K)$.
\end{enumerate}
\end{theorem}

\begin{proof}
We begin with the structure of the quotient.  Bekka and Ludwig prove that the
projective kernel of an irreducible unitary
representation of an exponential solvable Lie group is connected
\cite[Theorem~2.1]{BekkaLudwig1990}.  Since $K$ is also closed and normal,
$G/K$ is exponential by \cite[Corollary~1.8.5]{ArnalCurrey2020}.  Having
identified the quotient, we use the same representation-theoretic input to
obtain the precise decay statement needed
here: for each $w\in\cH$, the function
\[
 gK\longmapsto |\langle w,\pi(g)w\rangle|
\]
vanishes at infinity on $G/K$; see
\cite[Lemma~2.2.1, p.~520]{BekkaLudwig1990}.  This formulation is also used
explicitly in the proof of \cite[Proposition~2.2]{BeltitaVanVelthoven2024}.
Combining these cited statements gives (i)--(iii).
\end{proof}

Theorem~\ref{thm:projective-c0-input} is stated only for diagonal
coefficients, whereas the storage construction requires cross coefficients.
The following polarization argument passes from the former to the latter.

\begin{corollary}\label{cor:cross-c0}
Every irreducible unitary representation of an exponential solvable Lie group is
projectively $C_0$ in the sense of Definition~\ref{def:projective-c0}.
\end{corollary}

\begin{proof}
Fix $g\in G$ and set
\[
 B_g(u,v)=\langle u,\pi(g)v\rangle,
 \qquad q_g(w)=B_g(w,w).
\]
With our convention that the inner product is linear in the first variable,
direct expansion gives
\begin{equation}\label{eq:polarization}
\begin{aligned}
 B_g(u,v)=\frac14\bigl(&q_g(u+v)-q_g(u-v)\\
 &+i q_g(u+iv)-i q_g(u-iv)\bigr).
\end{aligned}
\end{equation}
Consequently,
\[
 |B_g(u,v)|
 \leq\frac14\bigl(
 |q_g(u+v)|+|q_g(u-v)|+|q_g(u+iv)|+|q_g(u-iv)|
 \bigr).
\]
By Theorem~\ref{thm:projective-c0-input}, every term on the right vanishes at
infinity on $G/K$.  Moreover, the modulus on the left is
well defined on $G/K$ and is continuous by strong continuity of $\pi$.
Consequently, it belongs to $C_0(G/K)$, as required.
\end{proof}

\section{The one-sided storage construction}\label{sec:storage}

\subsection{Superdecreasing weights}

The preceding section supplied the perturbation principle and explained how
projective coefficient decay will later produce weak-operator mixing.  We now
develop the analytic construction that converts such directional mixing into
a single orbit frame.  Its first ingredient is a sequence of weights
which separates the influence of the past from that of the future.  More
precisely, the construction must control two different accumulations of error.  The
sum $\sum a_n$ controls all information stored before a given block, whereas
$\sum N_nT_n^2$ controls the aggregate tail after that block.  The following
choice makes both quantities arbitrarily small while retaining the exact
normalization $N_na_n^2=1$.

\begin{lemma}[Superdecreasing weights]\label{lem:weights}
For every $\rho>0$, there exist integers $N_n\geq2$ and numbers
$a_n=N_n^{-1/2}$ such that, with
\begin{equation}\label{eq:weight-quantities}
 \sigma=\sum_{n\geq1}a_n,
 \qquad T_n=\sum_{m>n}a_m,
 \qquad \tau^2=\sum_{n\geq1}N_nT_n^2,
\end{equation}
one has $\sigma<\rho$ and $\tau<\rho$.
\end{lemma}

\begin{proof}
We construct the multiplicities recursively.  Choose numbers
$0<c_n\leq1/4$ satisfying
$4\sum_{n\geq1}c_n^2<\rho^2$.  First choose $N_1$ sufficiently large.
After $N_n$ has been chosen, choose the integer $N_{n+1}\geq2$ so large that
\begin{equation}\label{eq:weight-recursion}
 a_{n+1}=N_{n+1}^{-1/2}\leq c_na_n.
\end{equation}
The recursive inequality also controls each tail.  Indeed, because
$c_n\leq1/4$, iteration gives
$a_{n+k}\leq a_{n+1}4^{-(k-1)}$ for $k\geq1$.  Therefore
\begin{equation}\label{eq:tail-bound}
 T_n\leq \frac43a_{n+1}<2a_{n+1}.
\end{equation}
Using $N_n=a_n^{-2}$ and \eqref{eq:weight-recursion}, we obtain
\[
 N_nT_n^2
 <4\frac{a_{n+1}^2}{a_n^2}
 \leq4c_n^2.
\]
Summing the preceding estimates gives $\tau<\rho$.  The same geometric
argument yields $\sigma\leq(4/3)a_1$.  Choosing $N_1$ still larger if
necessary makes $\sigma<\rho$ and completes the proof.
\end{proof}

\begin{example}[An explicit superdecreasing pattern]\label{ex:explicit-weights}
Fix an integer $M\geq4$ and set
\[
 N_n=M^{2^n},\qquad a_n=M^{-2^{n-1}}.
\]
Then $a_{n+1}=a_n^2\leq M^{-1}a_n$.  Hence
$\sigma\leq a_1/(1-M^{-1})$, and
$T_n\leq a_{n+1}/(1-M^{-1})$.  Since $N_n=a_n^{-2}$,
\[
 N_nT_n^2
 \leq (1-M^{-1})^{-2}a_n^2.
\]
Moreover,
\[
 \sum_{n\geq1}a_n^2
 \leq \frac{a_1^2}{1-M^{-2}},
 \qquad
 \tau^2\leq
 \frac{a_1^2}{(1-M^{-1})^2(1-M^{-2})}.
\]
Consequently, both $\sigma$ and $\tau$ tend to zero as $M\to\infty$.  Although
the recursive proof of Lemma~\ref{lem:weights} is more flexible, this concrete
family makes clear why the block sizes must grow much faster than
geometrically.
\end{example}

\subsection{The model Parseval family}

The weights just constructed also determine the Parseval family that the
orbit will approximate.  Let $(e_n)_{n\geq1}$ be an orthonormal basis and suppose
$a_n=N_n^{-1/2}$.  Repeating $a_ne_n$ exactly $N_n$ times produces a
Parseval frame:
\begin{equation}\label{eq:model-parseval-calculation}
 \sum_{n\geq1}\sum_{j=1}^{N_n}
 |\langle f,a_ne_n\rangle|^2
 =\sum_{n\geq1}N_na_n^2|\langle f,e_n\rangle|^2
 =\|f\|^2.
\end{equation}
Thus multiplicity exactly compensates for the small norm of each atom.  The
storage theorem will approximate this repeated family by orbit vectors of a
single vector $\phi$.

\begin{example}[A finite toy block]\label{ex:finite-parseval-block}
In $\mathbb C^2$, four copies of $e_1/2$ together with nine copies of
$e_2/3$ form a Parseval frame.  Indeed, the four copies contribute
$|\langle f,e_1\rangle|^2$ and the nine copies contribute
$|\langle f,e_2\rangle|^2$.  The infinite construction below uses the same
normalization in every coordinate, while the rapidly growing multiplicities
also create room for a summable storage vector.
\end{example}

\subsection{Abstract storage theorem}

We are now in a position to formulate the abstract theorem.  Its hypotheses
isolate precisely the weak-operator limits used by the storage argument.
Recall that $T_m\to0$ in the weak operator topology
(abbreviated WOT) when $\langle u,T_mv\rangle\to0$ for every $u,v\in\cH$.

\begin{definition}[Local mixing-expansion datum]
\label{def:mixing-expansion}
For a strongly continuous unitary representation $\pi:G\to\U(\cH)$, a
\emph{local mixing-expansion datum} consists of a nondiscrete Hausdorff
topological group $A$, a continuous homomorphism $p:A\to G$, an identity
neighborhood $U\subset A$, and a sequence $(h_m)$ in $G$ such that
\[
 \pi(h_m)\xrightarrow{\mathrm{WOT}}0
\]
and, for every $a\in U\setminus\{e_A\}$,
\[
 \pi\bigl(h_mp(a)h_m^{-1}\bigr)\xrightarrow{\mathrm{WOT}}0.
\]
\end{definition}

Before constructing the frame, we record a simple consequence of the first
weak-operator limit.  It ensures that the recursively chosen sampling points
can avoid every previously used projective coset.

\begin{lemma}[Weak-operator decay forces projective escape]
\label{lem:wot-projective-escape}
Let $G$ be a second-countable locally compact group, let
$\pi:G\to\U(\cH)$ be strongly continuous, and put $K=\pker(\pi)$.  If
$\pi(g_m)\to0$ in the weak operator topology, then
\[
  g_mK\longrightarrow\infty\qquad\text{in }G/K.
\]
\end{lemma}

\begin{proof}
Suppose, to the contrary, that the asserted escape fails.  Then some compact
subset of $G/K$ contains a subsequence of $(g_mK)$.
The quotient is second countable and locally compact, hence metrizable; after
passing to a further subsequence, suppose $g_mK\to gK$.  For $0\neq v\in\cH$,
the scalar action of $K$ and strong continuity give the continuous function
\[
 xK\longmapsto |\langle\pi(g)v,\pi(x)v\rangle|.
\]
Its values at $g_mK$ tend to $\|v\|^2$, whereas weak convergence of
$\pi(g_m)$ to zero forces the same values to tend to zero.  This
contradiction proves the lemma.
\end{proof}

\begin{theorem}[Sequential weak-operator storage]
\label{thm:abstract-storage}
Let $G$ be a second-countable locally compact group, let
$\pi:G\to\U(\cH)$ be a strongly continuous unitary representation on an
infinite-dimensional separable Hilbert space, and put $K=\pker(\pi)$.
Assume that $\pi$ admits a local mixing-expansion datum
$(A,p,U,(h_m))$ in the sense of Definition~\ref{def:mixing-expansion}.
Then, for every $0<\varepsilon<1$, there exist a nonzero vector
$\phi\in\cH$ and a countable set $\Gamma\subset G$ such that
\begin{equation}\label{eq:storage-frame}
 (1-\varepsilon)^2\|f\|^2
 \leq \sum_{\gamma\in\Gamma}
       |\langle f,\pi(\gamma)\phi\rangle|^2
 \leq (1+\varepsilon)^2\|f\|^2
 \qquad(f\in\cH).
\end{equation}
Moreover, $\Gamma$ may be chosen so that the quotient map
$q:G\to G/K$ is injective on $\Gamma$ and $q(\Gamma)$ is left relatively
separated (hence closed and discrete).
\end{theorem}

Two features of the theorem should be emphasized.  First, its hypotheses are
directional and require neither irreducibility nor a global $C_0$ property.
Second, the construction produces the frame and the projectively separated
sampling set simultaneously.  Before proving the theorem, it is useful to
identify the target and the
three errors that will appear in every retrieved vector.  The $n$th basis
vector is stored as $a_n\pi(b_n)e_n$.  Applying
$\pi(p_{n,j}b_n^{-1})$ produces the target $a_ne_n$ plus a small local
retrieval error.  The terms stored before stage $n$ form the \emph{past
interference}; they are controlled collectively by a Gram-matrix estimate.
The terms stored after stage $n$ form the \emph{future interference}; they
are controlled by the superdecreasing tail.  These are different estimates,
which is why the numerical lemma tracks both $\sigma$ and $\tau$.

\begin{figure}[t]
\centering
\resizebox{\linewidth}{!}{%
\begin{tikzpicture}[
  font=\small,
  >={Latex[length=2.3mm]},
  store/.style={draw=OrbitInk, very thick, rounded corners=2pt,
    fill=OrbitMist, align=center, text width=7.0cm, minimum height=1.15cm},
  atom/.style={draw, thick, rounded corners=2pt, align=center,
    text width=3.05cm, minimum height=1.25cm},
  flow/.style={->,very thick,draw=OrbitInk}
]
\node[store] (store) at (0,2.0)
  {$\displaystyle \phi=\sum_{m\geq1}a_m\pi(b_m)e_m$\par
   one vector stores every basis direction};
\node[store,fill=OrbitBlue!7,draw=OrbitBlue] (retrieve) at (0,0.35)
  {$\displaystyle \pi(\gamma_{n,j})\phi,\qquad
    \gamma_{n,j}=p_{n,j}b_n^{-1}$};
\node[atom,draw=OrbitTeal,fill=OrbitTeal!8] (target) at (-5.1,-1.75)
  {target\par $a_ne_n$\par exact Parseval energy};
\node[atom,draw=OrbitGold,fill=OrbitGold!10] (local) at (-1.7,-1.75)
  {local error\par $q_{n,j}$\par $\|Q_0\|<d$};
\node[atom,draw=OrbitRed,fill=OrbitRed!7] (past) at (1.7,-1.75)
  {past interference\par $r_{n,j}$\par
   $\|R_-\|\leq\sqrt{\sigma^2+\eta^2}$};
\node[atom,draw=OrbitBlue,fill=OrbitBlue!7] (future) at (5.1,-1.75)
  {future tail\par $z_{n,j}$\par $\|R_+\|\leq\tau$};
\draw[flow] (store) -- node[right,font=\scriptsize]
  {retrieve one block} (retrieve);
\foreach \x in {target,local,past,future}
  \draw[flow] (retrieve.south) -- (\x.north);
\node[draw=OrbitInk,very thick,rounded corners=2pt,fill=OrbitMist,
  align=center,text width=8.2cm,minimum height=.85cm] (sum) at (0,-3.55)
  {$\displaystyle \|D-D_0\|\leq
    d+\sqrt{\sigma^2+\eta^2}+\tau<\varepsilon$};
\draw[flow] (local.south) -- (sum.north west);
\draw[flow] (past.south) -- (sum.north);
\draw[flow] (future.south) -- (sum.north east);
\end{tikzpicture}
}
\caption{Storage and retrieval at block $(n,j)$.  The target columns form a
Parseval frame; three analytically different error operators are controlled
in norm before the perturbation lemma is applied.}
\label{fig:storage-retrieval}
\end{figure}
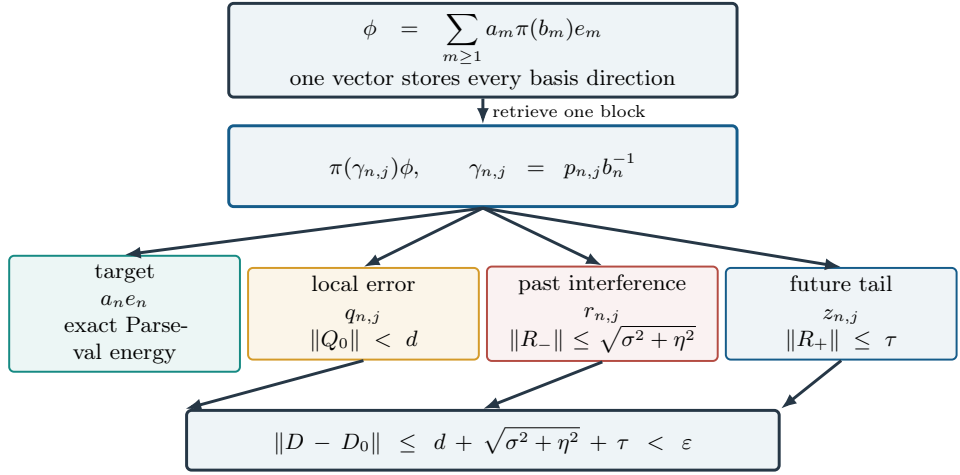

\begin{proof}
We divide the proof into eight steps.  Steps~1--2 choose the Parseval target
and the local retrieval points.  Step~3 places the storage blocks recursively,
and Step~4 turns the resulting Gram budgets into a Bessel bound for past
interference.  Step~5 constructs the storage vector and the exact retrieval
decomposition, while Step~6 controls the remaining local and future errors.
Finally, Steps~7--8 compare the synthesis operators and verify the geometry
of the sampling set.

Fix $0<\varepsilon<1$ and put
\begin{equation}\label{eq:rho-choice}
   \rho=\frac{\varepsilon}{5}.
\end{equation}
Choose an orthonormal basis $(e_n)_{n\geq1}$ of $\cH$, and use
Lemma~\ref{lem:weights} to obtain $N_n$, $a_n$, $\sigma$, $T_n$, and $\tau$.

\Step{Step 1: choose the target Parseval frame}
We begin by fixing the family that will be approximated.  Let
\[
 \cI=\{(n,j):n\geq1,\ 1\leq j\leq N_n\}
\]
and define
\begin{equation}\label{eq:target-parseval}
 u_{n,j}=a_ne_n\qquad((n,j)\in\cI).
\end{equation}
By \eqref{eq:model-parseval-calculation}, $(u_{n,j})_{(n,j)\in\cI}$ is a
Parseval frame.  Write $D_0:\ell^2(\cI)\to\cH$ for its synthesis operator.

\Step{Step 2: choose small retrieval perturbations}
Having fixed the target frame, we next choose the local points that will
retrieve each target column.  Choose positive numbers $\delta_n$ for which
\begin{equation}\label{eq:d-definition}
 d^2:=\sum_{n\geq1}\delta_n^2<\rho^2.
\end{equation}
Choose a symmetric identity neighborhood $V\subset A$ with
$V^{-1}V\subset U$.  Since $p(e_A)=e_G$ and $a\mapsto\pi(p(a))e_n$ is
norm-continuous at $e_A$, there is an identity neighborhood $J_n\subset V$
on which $\|\pi(p(a))e_n-e_n\|<\delta_n$.  A nondiscrete Hausdorff
topological group has no isolated points, so $J_n$ contains $N_n$ distinct
points
\begin{equation}\label{eq:retrieval-points}
 a_{n,j}\in J_n,\qquad 1\leq j\leq N_n.
\end{equation}
Writing $p_{n,j}=p(a_{n,j})$, the required local approximation becomes
\begin{equation}\label{eq:retrieval-small}
 \|\pi(p_{n,j})e_n-e_n\|<\delta_n.
\end{equation}

\Step{Step 3: place the storage blocks recursively}
We now place the storage blocks so that both their mutual interactions and
their projective cosets remain under control.  Fix a bijection
$\iota:\cI\to\N$, choose $0<\eta<\rho$, and set $M_0=0$.
Suppose that $M_1,\ldots,M_{n-1}$ have been chosen.  Put
\begin{equation}\label{eq:past-storage-vector}
 b_m=h_{M_m},
 \qquad
 w_n=\sum_{m<n}a_m\pi(b_m)e_m,
\end{equation}
and, for a tentative integer $r\geq1$, define
\begin{equation}\label{eq:past-column-tentative}
 r_{n,j}(r)=\pi\bigl(p_{n,j}h_r^{-1}\bigr)w_n.
\end{equation}
The purpose of the recursion is to choose $M_n>M_{n-1}$ so that
\begin{align}
 |\langle r_{n,j}(M_n),r_{\ell,k}(M_\ell)\rangle|
 &\leq \eta^2 2^{-\iota(n,j)-\iota(\ell,k)},
 &&\ell<n,                                      \label{eq:cross-budget}\\
 |\langle r_{n,j}(M_n),r_{n,k}(M_n)\rangle|
 &\leq \eta^2 2^{-\iota(n,j)-\iota(n,k)},
 &&j\neq k.                                     \label{eq:internal-budget}
\end{align}
At stage $n$ only finitely many inequalities are involved.  To see that they
can all be satisfied, we use unitarity together with the convention that the
inner product is linear in the first variable and obtain
\begin{align}
 \langle r_{n,j}(r),r_{n,k}(r)\rangle
 &=\left\langle w_n,
 \pi\bigl(h_rp_{n,j}^{-1}p_{n,k}h_r^{-1}\bigr)w_n
 \right\rangle,                                  \label{eq:internal-gram}\\
 \langle r_{n,j}(r),r_{\ell,k}(M_\ell)\rangle
 &=\left\langle w_n,
 \pi\bigl(h_rp_{n,j}^{-1}p_{\ell,k}b_\ell^{-1}\bigr)w_\ell
 \right\rangle.                                  \label{eq:cross-gram}
\end{align}
If $j\neq k$, then
$a_{n,j}^{-1}a_{n,k}\in V^{-1}V\setminus\{e_A\}\subset U\setminus\{e_A\}$.
Because $p$ is a homomorphism, the conjugated WOT limit in
Definition~\ref{def:mixing-expansion} makes the coefficient in
\eqref{eq:internal-gram} tend to zero.  In \eqref{eq:cross-gram}, everything
after $h_r$ is fixed, so the first WOT limit in that definition, with the
fixed right factor absorbed into the second vector, makes that coefficient
tend to zero.  Intersecting the finitely many eventual tails of $r$ therefore
gives both coefficient estimates.

At the same stage we impose one additional finite requirement, namely the
distinctness of the projective cosets.  For a tentative $r$, put
\begin{equation}\label{eq:sampling-elements-tentative}
 \gamma_{n,j}(r)=p_{n,j}h_r^{-1}.
\end{equation}
We require the cosets $\gamma_{n,j}(M_n)K$ to be distinct from every
previously chosen coset and from one another.  Lemma
\ref{lem:wot-projective-escape} gives $h_rK\to\infty$.  For fixed $j$,
$p_{n,j}h_r^{-1}K\to\infty$, because inversion and multiplication by a fixed
element are proper homeomorphisms of $G/K$; hence finitely many earlier points
may be avoided.  It remains only to distinguish points within the new block.
If the $j$th and $k$th cosets in that block were equal, then
\[
 h_rp_{n,j}^{-1}p_{n,k}h_r^{-1}\in K.
\]
The conjugated WOT limit and Lemma~\ref{lem:wot-projective-escape} show that the coset
of the element on the left escapes, so this is impossible for all sufficiently
large $r$.  Consequently, the coefficient budgets and all distinctness
requirements can be met simultaneously.

Choose such an $M_n$ and set
\begin{equation}\label{eq:sampling-elements}
 b_n=h_{M_n},\qquad
 \gamma_{n,j}=\gamma_{n,j}(M_n)=p_{n,j}b_n^{-1},\qquad
 r_{n,j}=r_{n,j}(M_n).
\end{equation}

\Step{Step 4: bound all past interference by Schur's test}
The recursion controls every off-diagonal Gram entry.  We now combine these
estimates to control the entire past-interference operator.  The diagonal
Gram entries satisfy
\begin{equation}\label{eq:past-column-norm}
 \|r_{n,j}\|^2=\|w_n\|^2
 \leq\left(\sum_{m<n}a_m\right)^2
 \leq\sigma^2.
\end{equation}
Every unordered off-diagonal pair appears in either
\eqref{eq:cross-budget} or \eqref{eq:internal-budget} when its later block is
chosen.  Hermitian symmetry controls the reverse order.  Hence, for each
$\alpha\in\cI$,
\begin{align}\label{eq:past-row-sum}
 \sum_{\substack{\beta\in\cI\\\beta\neq\alpha}}
 |\langle r_\alpha,r_\beta\rangle|
 &\leq \eta^2 2^{-\iota(\alpha)}
       \sum_{\beta\in\cI}2^{-\iota(\beta)}
 \leq\eta^2.
\end{align}
Applying Schur's test to every finite Gram section gives a uniform bound.
Consequently, for every finitely supported $c$,
\[
 \left\|\sum_{\alpha\in\cI}c_\alpha r_\alpha\right\|^2
 \leq(\sigma^2+\eta^2)\|c\|_2^2.
\]
It follows that the family $(r_{n,j})$ is Bessel and that its synthesis
operator $R_-$ satisfies
\begin{equation}\label{eq:past-synthesis-bound}
 \|R_-\|\leq(\sigma^2+\eta^2)^{1/2}.
\end{equation}

\Step{Step 5: store all basis vectors in one vector and retrieve them}
Having controlled all previously stored terms, we can now assemble the
storage vector.  Define
\begin{equation}\label{eq:storage-vector}
 \phi=\sum_{m\geq1}a_m\pi(b_m)e_m.
\end{equation}
This series converges absolutely because $\sum_ma_m=\sigma<\infty$.  Applying
$\pi(\gamma_{n,j})$ and separating the terms with $m<n$, $m=n$, and $m>n$
gives the exact identity
\begin{equation}\label{eq:retrieval-decomposition}
 \pi(\gamma_{n,j})\phi
 =a_n\pi(p_{n,j})e_n+r_{n,j}+z_{n,j},
\end{equation}
where
\begin{equation}\label{eq:future-column}
 z_{n,j}
 =\sum_{m>n}a_m
   \pi\bigl(p_{n,j}b_n^{-1}b_m\bigr)e_m.
\end{equation}
The series defining $\phi$ and $z_{n,j}$ are absolutely convergent.  In
particular, no rearrangement of a conditionally convergent series occurs in
the preceding calculation.

\Step{Step 6: estimate the main and future errors}
It remains to estimate the two parts of the retrieval decomposition not
covered by $R_-$.  The main-term error columns are
\[
 q_{n,j}=a_n\bigl(\pi(p_{n,j})e_n-e_n\bigr).
\]
They define a Hilbert--Schmidt synthesis operator $Q_0$ satisfying
\begin{equation}\label{eq:main-error-bound}
\begin{aligned}
 \|Q_0\|^2
 &\leq\|Q_0\|_{\HS}^2
   =\sum_{n\geq1}\sum_{j=1}^{N_n}\|q_{n,j}\|^2 \\
 &<\sum_{n\geq1}N_na_n^2\delta_n^2
   =\sum_{n\geq1}\delta_n^2=d^2.
\end{aligned}
\end{equation}
For the future terms, unitarity and the triangle inequality give
$\|z_{n,j}\|\leq T_n$.  Hence
the future columns define a Hilbert--Schmidt synthesis operator $R_+$ with
\begin{equation}\label{eq:future-error-bound}
 \|R_+\|^2
 \leq\|R_+\|_{\HS}^2
 \leq\sum_{n\geq1}N_nT_n^2=\tau^2.
\end{equation}

\Step{Step 7: compare the actual and model synthesis operators}
We now combine the three error operators and compare the resulting orbit
family with the model Parseval frame.  Let
$D_{\mathrm{fin}}:c_{00}(\cI)\to\cH$ be the algebraic synthesis map for the
orbit family $(\pi(\gamma_{n,j})\phi)$.  Equations
\eqref{eq:retrieval-decomposition}--\eqref{eq:future-error-bound} give
\begin{equation}\label{eq:synthesis-decomposition}
 D_{\mathrm{fin}}=D_0+Q_0+R_-+R_+
 \qquad\text{on }c_{00}(\cI).
\end{equation}
Every operator on the right is bounded on $\ell^2(\cI)$.  Therefore
$D_{\mathrm{fin}}$ extends uniquely to the synthesis operator
$D:\ell^2(\cI)\to\cH$ of the orbit family, and
\begin{equation}\label{eq:total-error}
 \|D-D_0\|
 \leq d+(\sigma^2+\eta^2)^{1/2}+\tau=:E.
\end{equation}
Since $d,\sigma,\eta,\tau<\rho$,
\begin{equation}\label{eq:error-budget-final}
 E<(2+\sqrt2)\rho
   =\frac{2+\sqrt2}{5}\,\varepsilon
   <\varepsilon<1.
\end{equation}
The perturbation estimate is now small enough to apply
Lemma~\ref{lem:synthesis-perturbation}.  It first gives the stronger bounds
$(1-E)^2$ and $(1+E)^2$, and hence the bounds in \eqref{eq:storage-frame}.
In particular, $\phi\neq0$; otherwise every orbit vector would vanish.

\Step{Step 8: prove relative separation in the quotient}
It remains to verify the promised geometry of the sampling set.  Let
\[
 \Gamma=\{\gamma_{n,j}:(n,j)\in\cI\}.
\]
The recursive distinctness requirement shows that the quotient map
$q:G\to G/K$ is injective on $\Gamma$.  The orbit family is Bessel, with
upper frame bound at most $(1+E)^2$.  Since $\phi\neq0$,
Lemma~\ref{lem:bessel-relative-separation} implies that $q(\Gamma)$ is left
relatively separated.  In particular, $q(\Gamma)$ is locally finite and
therefore closed and discrete.  This completes the proof of all assertions.
\end{proof}

The preceding theorem is formulated only in terms of weak-operator limits.
In applications, these limits are most conveniently verified by combining
projective coefficient decay with geometric escape, as the next corollary
records.

\begin{corollary}[Projective $C_0$ escape criterion]
\label{cor:projective-storage}
Let $G$ be a second-countable locally compact group, let
$\pi:G\to\U(\cH)$ be a strongly continuous unitary representation on an
infinite-dimensional separable Hilbert space, and put $K=\pker(\pi)$.
Let $A$ be a nondiscrete Hausdorff topological group, let
$p:A\to G$ be a continuous homomorphism, let $U\subset A$ be an identity
neighborhood, and let $(h_m)$ be a sequence in $G$.  Suppose that $\pi$ is
projectively $C_0$ and that
\begin{align}
 h_mK&\longrightarrow\infty,                                      \label{eq:geometric-storage-escape}\\
 h_mp(a)h_m^{-1}K&\longrightarrow\infty
 \qquad(a\in U\setminus\{e_A\}).                                \label{eq:geometric-conjugation-escape}
\end{align}
Then all conclusions of Theorem~\ref{thm:abstract-storage} hold.
\end{corollary}

\begin{proof}
Let $u,v\in\cH$.  Projective $C_0$ decay applied to
\eqref{eq:geometric-storage-escape} gives
$\langle u,\pi(h_m)v\rangle\to0$.  Applied to
\eqref{eq:geometric-conjugation-escape}, it gives
\[
 \langle u,\pi(h_mp(a)h_m^{-1})v\rangle\longrightarrow0
 \qquad(a\neq e_A).
\]
Thus both WOT hypotheses of Theorem~\ref{thm:abstract-storage} are satisfied,
and the conclusion follows from that theorem.
\end{proof}

\begin{remark}[Equal orbit norms and vanishing model norms]
\label{rem:equal-norm-objection}
All actual orbit vectors $\pi(\gamma_{n,j})\phi$ have the same norm
$\|\phi\|$, whereas the model vectors $a_ne_n$ have norms $a_n\to0$.
There is no contradiction: the construction also makes
\[
 \|\phi\|\leq\sum_{m\geq1}a_m=\sigma<\rho,
\]
so the common norm of every orbit vector lies inside the global perturbation
budget.  The normalization $N_na_n^2=1$ places the required energy in the
multiplicity $N_n$, not in the norm of any single model atom.  The comparison
is therefore an operator-norm comparison of synthesis maps, rather than a
uniform relative comparison of corresponding column norms.
\end{remark}

\begin{remark}[The three error mechanisms]\label{rem:three-errors}
The decomposition \eqref{eq:synthesis-decomposition} separates three
analytically different errors.  The operator $Q_0$ is the local retrieval
error and is small by continuity near the identity.  The operator $R_-$ is
past interference and is controlled by directional weak-operator decay plus
a Schur estimate.  The operator $R_+$ is future interference and is controlled
only by the superdecreasing weights.  Treating these mechanisms separately
is what prevents the infinitely many retrieval conditions from competing for
one uniform pointwise estimate.
\end{remark}

\section{Conjugation escape in exponential solvable Lie groups}
\label{sec:geometry}

We now provide the geometric input required by
Corollary~\ref{cor:projective-storage}.  More precisely, we construct two
one-parameter subgroups whose conjugation separates every pair of distinct
retrieval parameters.  This is the only point at which noncommutativity and
the spectral characterization of exponentiality interact.  Recall that, for
an exponential solvable Lie algebra, $\ad Z$ has no nonzero purely imaginary
eigenvalue for any $Z$; see
\cite[Chapter~1, Theorem~1]{LeptinLudwig1994}.

\begin{lemma}[Adjoint escape]\label{lem:lie-escape}
Let $Q$ be a nonabelian exponential solvable Lie group with Lie algebra
$\mathfrak q$.  There exist nonzero $X_0,Y_0\in\mathfrak q$ such that
\begin{equation}\label{eq:ad-escape}
 \|\exp(R\ad X_0)Y_0\|\longrightarrow\infty
 \qquad(R\to+\infty).
\end{equation}
The assertion is independent of the chosen norm on the finite-dimensional
space $\mathfrak q$.
\end{lemma}

\begin{proof}
Since $\mathfrak q$ is nonabelian, choose $X_0$ such that
$A=\ad X_0\neq0$.  We distinguish two cases according to the spectrum of
$A$.

\Step{Step 1: an eigenvalue with nonzero real part}
Suppose that the complexification of $A$ has a nonzero eigenvalue.
Exponentiality excludes nonzero purely imaginary eigenvalues.  After
replacing $X_0$ by $-X_0$, if necessary, the sum of the generalized
eigenspaces with positive real part is nonzero and is stable under complex
conjugation.  Let $E$ be its corresponding nonzero real $A$-invariant
subspace.  Jordan normal form on the complexification gives
\begin{equation}\label{eq:contracting-inverse}
 \|\exp(-RA)|_E\|\longrightarrow0.
\end{equation}
Choose $0\neq Y_0\in E$.  Then
\[
 \|Y_0\|
 \leq \|\exp(-RA)|_E\|\,\|\exp(RA)Y_0\|.
\]
Since the first factor on the right tends to zero, the second must tend to
infinity, and the desired conclusion follows in this case.

\Step{Step 2: a nilpotent Jordan block}
We now consider the remaining case.  Suppose that every eigenvalue of $A$ is
zero.  Then $A$ is nonzero and
nilpotent.  Choose $Y_0$ and $k\geq1$ such that
$A^kY_0\neq0$ and $A^{k+1}Y_0=0$.  Thus
\begin{equation}\label{eq:nilpotent-growth}
 \exp(RA)Y_0=\sum_{j=0}^{k}\frac{R^j}{j!}A^jY_0.
\end{equation}
Choose a linear functional $\ell$ with $\ell(A^kY_0)\neq0$.  Applying
$\ell$ to \eqref{eq:nilpotent-growth} produces a scalar polynomial of degree
$k$ with nonzero leading coefficient.  Its absolute value tends to infinity,
and therefore so does $\|\exp(RA)Y_0\|$.

The two cases exhaust the spectrum of $A$ and therefore prove the lemma.
\end{proof}

The preceding Lie-algebra statement becomes an exact group-level escape
identity because the exponential map of $Q$ is a global diffeomorphism.

\begin{proposition}[The escape package]
\label{prop:geometric-escape}
Let $Q$ be a nonabelian exponential solvable Lie group.  There exist continuous
one-parameter subgroups
\[
 \overline b(R)=\exp_Q(RX_0),
 \qquad
 \overline p(t)=\exp_Q(tY_0)
\]
such that
\begin{align}
 \overline b(R)&\longrightarrow\infty,
 &&R\to+\infty,                                  \label{eq:b-escape-Q}\\
 \overline b(R)\overline p(s)^{-1}\overline p(t)
 \overline b(R)^{-1}&\longrightarrow\infty,
 &&R\to+\infty,\quad s\neq t.                  \label{eq:conj-escape-Q}
\end{align}
\end{proposition}

\begin{proof}
Choose $X_0,Y_0$ as in Lemma~\ref{lem:lie-escape}.  We verify separately the
escape of the storage curve and that of the conjugated retrieval curve.

\Step{Step 1: escape of the storage direction}
Because $X_0\neq0$, the curve $RX_0$ leaves every compact subset of
$\mathfrak q$.  The map $\exp_Q:\mathfrak q\to Q$ is a homeomorphism, so
$\overline b(R)=\exp_Q(RX_0)$ leaves every compact subset of $Q$.

\Step{Step 2: exact computation of the conjugate}
Having dealt with the storage direction, let $s\neq t$.  The one-parameter
group law and the standard conjugation
identity give
\begin{align}\label{eq:exact-conjugation}
 &\overline b(R)\overline p(s)^{-1}\overline p(t)
   \overline b(R)^{-1}\notag\\
 &\qquad
 =\exp_Q\bigl((t-s)\exp(R\ad X_0)Y_0\bigr).
\end{align}
Indeed, $\overline p(s)^{-1}\overline p(t)=\overline p(t-s)$ and
$\Ad(\exp_Q(RX_0))=\exp(R\ad X_0)$.

\Step{Step 3: convert Lie-algebra escape into group escape}
It remains to interpret the exact identity.  By
Lemma~\ref{lem:lie-escape}, the argument of $\exp_Q$ in
\eqref{eq:exact-conjugation} escapes in $\mathfrak q$.  Since $\exp_Q$ is a
homeomorphism, its image escapes in $Q$.  This proves both assertions.
\end{proof}

The next two examples illustrate the two spectral alternatives used in the
proof.  The affine group exhibits exponential expansion, while the
Heisenberg algebra exhibits the polynomial growth produced by a nilpotent
Jordan block.

\begin{example}[Exponential escape in the affine group]
\label{ex:affine-escape}
Write the connected affine group of the line as $Q=\R\ltimes\R$ with
\[
 (a,b)(a',b')=(a+a',b+\e^a b').
\]
Its Lie algebra has a basis $X,Y$ with $[X,Y]=Y$, and the exponential map is
a global diffeomorphism.  Since $\ad X(Y)=Y$,
\[
 \exp(R\ad X)Y=\e^R Y.
\]
At the group level,
\[
 (R,0)(0,t)(-R,0)=(0,\e^Rt).
\]
Consequently, every nontrivial translation is driven to infinity by positive
dilation, which is precisely the expanding spectral case in
Lemma~\ref{lem:lie-escape}.
\end{example}

\begin{example}[Polynomial escape in the Heisenberg algebra]
\label{ex:nilpotent-escape}
Let $\mathfrak q$ have basis $X,Y,Z$ with $[X,Y]=Z$ and $Z$ central.  Then
$(\ad X)Y=Z$ and $(\ad X)^2Y=0$, so
\[
 \exp(R\ad X)Y=Y+RZ.
\]
In the simply connected Heisenberg group,
\[
 \exp_Q(RX)\exp_Q(tY)\exp_Q(-RX)
 =\exp_Q(tY+tRZ).
\]
Thus the curve escapes polynomially.  The rate is irrelevant because
projective $C_0$
decay only requires eventual escape from compact sets.  For a Schr\"odinger
representation the center belongs to the projective kernel, so the effective
quotient is abelian; that particular representation is handled by
Section~\ref{sec:abelian}, not by the nonabelian storage branch.
\end{example}

\begin{figure}[t]
\centering
\begin{tikzpicture}
\begin{semilogyaxis}[
  width=.88\linewidth,
  height=.43\linewidth,
  xmin=0,xmax=6,
  ymin=1,ymax=500,
  domain=0:6,
  samples=160,
  axis line style={draw=OrbitInk},
  tick style={draw=OrbitInk},
  tick label style={font=\small,text=OrbitInk},
  label style={font=\small,text=OrbitInk},
  xlabel={storage parameter $R$},
  ylabel={model displacement from the identity},
  grid=both,
  major grid style={draw=OrbitInk!18},
  minor grid style={draw=OrbitInk!8},
  clip=false
]
\addplot[very thick,OrbitBlue] {exp(x)}
  node[pos=.82,above left,text=OrbitBlue,font=\small]
  {affine: $\e^R$};
\addplot[very thick,OrbitGold] {sqrt(1+x^2)}
  node[pos=.80,below right,text=OrbitGold,font=\small]
  {nilpotent: $\sqrt{1+R^2}$};
\addplot[densely dashed,OrbitRed,thick,domain=0:6] {5}
  node[pos=.97,above left,text=OrbitRed,font=\scriptsize]
  {any fixed compact scale};
\end{semilogyaxis}
\end{tikzpicture}
\caption{The two model escape mechanisms on a logarithmic vertical scale.
The storage proof uses no rate estimate: exponential and polynomial growth
are equally effective once every fixed compact scale is eventually crossed.}
\label{fig:escape-rates}
\end{figure}
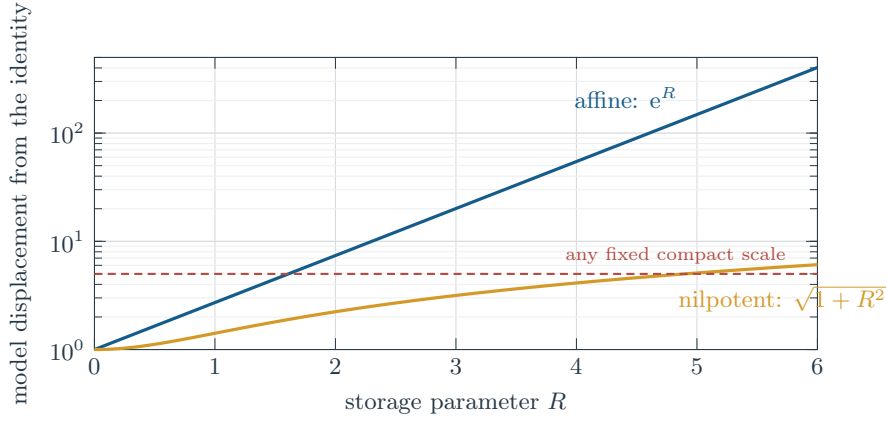

We now return to an irreducible representation of the original group $G$.
The quotient dictionary allows the two escape directions in $Q$ to be lifted
to $G$, while projective $C_0$ decay converts their geometric escape into the
weak-operator limits required by the storage theorem.

\begin{corollary}[Lifting the escape directions]
\label{cor:lifted-escape}
Let $G$ be an exponential solvable Lie group, let $\pi$ be irreducible, set
$K=\pker(\pi)$, and suppose that $Q=G/K$ is nonabelian.  Then there exist
a continuous homomorphism $p:\R\to G$ and a sequence $(h_m)$ in $G$
satisfying the hypotheses of Theorem~\ref{thm:abstract-storage}.
\end{corollary}

\begin{proof}
By Theorem~\ref{thm:projective-c0-input}, the quotient $Q$ is exponential.
We may therefore choose
$X_0,Y_0\in\mathfrak q$ as in Proposition~\ref{prop:geometric-escape} and
choose lifts $X,Y\in\mathfrak g$ under $dq_e$.  Define
\[
 b(R)=\exp_G(RX),
 \qquad p(t)=\exp_G(tY).
\]
By \eqref{eq:quotient-exp-intertwining}, their images in $Q$ are
$\overline b(R)$ and $\overline p(t)$.  Equations
\eqref{eq:b-escape-Q} and \eqref{eq:conj-escape-Q} give the geometric escape
hypotheses of Corollary~\ref{cor:projective-storage} for $h_m=b(m)$ and
$A=U=(\R,+)$.  Finally, Corollary~\ref{cor:cross-c0} supplies projective
$C_0$ decay, so both WOT hypotheses of
Theorem~\ref{thm:abstract-storage} follow.
\end{proof}

\section{Beyond exponential solvable Lie groups}
\label{sec:beyond-exponential}

The directional formulation separates the storage argument from the source
of weak mixing.  We now exploit this separation in two nonsolvable settings.
In both cases Howe--Moore decay supplies the analytic input, while
root-subgroup expansion supplies the geometry.

\subsection{Connected semisimple real Lie groups}

\begin{proof}[Proof of Theorem~\ref{thm:simple-main}]
We first identify an active noncompact factor.  Some connected noncompact
almost-simple normal factor $S$ of $G$ acts
nontrivially.  Indeed, otherwise $\pi$ would factor through the compact
semisimple quotient obtained from the compact factors (up to a finite central
quotient), whose irreducible unitary representations are finite-dimensional,
contrary to the hypothesis.  The subspace of $S$-fixed vectors
is $G$-invariant; irreducibility therefore shows that it is zero.  The
Howe--Moore theorem applied to $\pi|_S$ implies that
\begin{equation}\label{eq:howe-moore-real}
 \pi(g_m)\xrightarrow{\mathrm{WOT}}0
 \qquad\text{whenever }g_m\longrightarrow\infty\text{ in }S;
\end{equation}
see \cite{HoweMoore1979}.  This provides the required decay once appropriate
escaping sequences have been constructed.

We next construct those sequences from a restricted root.  Put
$\mathfrak s=\operatorname{Lie}(S)$ and choose a Cartan decomposition
$\mathfrak s=\mathfrak k_S\oplus\mathfrak p_S$, a maximal abelian subspace
$\mathfrak a\subset\mathfrak p_S$, and a nonzero
restricted root $\alpha$.  Choose $H\in\mathfrak a$ with $\alpha(H)>0$ and
$0\neq Y\in\mathfrak s_\alpha$.  With $A=(\R,+)$ and $U=A$, set
\[
 p(t)=\exp_G(tY),\qquad h_m=\exp_G(mH).
\]
The split torus $\exp_S(\mathfrak a)$ and the root subgroup
$\exp_G(\R Y)$ are closed, and their displayed parameters are proper; see
\cite[Chapter~VI]{Knapp2002}.  Hence $h_m\to\infty$.  Moreover,
\begin{equation}\label{eq:restricted-root-expansion}
 h_mp(t)h_m^{-1}
 =\exp_G\bigl(\e^{m\alpha(H)}tY\bigr)\longrightarrow\infty
 \qquad(t\neq0).
\end{equation}
Applying \eqref{eq:howe-moore-real} to the storage sequence and to every
nontrivial conjugated retrieval sequence shows that
$(A,p,(h_m))$ satisfies the hypotheses of
Theorem~\ref{thm:abstract-storage}.  The conclusion of that theorem proves
the result.
\end{proof}

\subsection{Special linear groups over arbitrary local fields}

\begin{proof}[Proof of Theorem~\ref{thm:local-field-main}]
We follow the same scheme, but now every ingredient can be written in matrix
coordinates.  Write $G=\operatorname{SL}_d(\mathbb k)$.  The representation is nontrivial
and has no nonzero invariant vector.  The Howe--Moore theorem over local
fields therefore gives
\begin{equation}\label{eq:howe-moore-local-field}
 \pi(g_m)\xrightarrow{\mathrm{WOT}}0
 \qquad\text{whenever }g_m\longrightarrow\infty\text{ in }G;
\end{equation}
see \cite{HoweMoore1979,Ciobotaru2014}.  It remains to exhibit an expanding
root subgroup.

Choose $c\in\mathbb k^\times$ with $|c|>1$, and let $E_{12}$ denote the
elementary matrix.  Set
\begin{equation}\label{eq:local-field-expanding-pair}
 h_m=\operatorname{diag}(c^m,c^{-m},1,\ldots,1),
 \qquad p(t)=I+tE_{12}\quad(t\in\mathbb k).
\end{equation}
The sequence $h_m$ leaves every compact subset of $G$, and direct
multiplication gives
\begin{equation}\label{eq:local-field-root-expansion}
 h_mp(t)h_m^{-1}=p(c^{2m}t).
\end{equation}
For $t\neq0$, the absolute value of the $(1,2)$ entry tends to infinity, so
the sequence on the right leaves every compact subset of $G$.  Taking
$A=(\mathbb k,+)$ and $U=A$, with $A$ nondiscrete, equations
\eqref{eq:howe-moore-local-field}--\eqref{eq:local-field-root-expansion}
give a local mixing-expansion datum.  Applying
Theorem~\ref{thm:abstract-storage} completes the proof.
\end{proof}

\begin{remark}[Reach and boundary of the extension]
The same proof extends to a split connected absolutely almost simple group
over a local field whenever the Howe--Moore property and properness of a
nontrivial root homomorphism are available; replacing $I+tE_{12}$ by that
root homomorphism then supplies the datum.  We keep the headline local-field
statement at $\operatorname{SL}_d$, where both inputs are explicit.
Theorem~\ref{thm:abstract-storage} also applies
representation by representation to nonexponential solvable groups whenever
the two directional WOT limits can be verified.  No blanket assertion for all
such solvable groups is made: recurrent adjoint rotations and failure of
projective coefficient decay can obstruct the present mechanism.
\end{remark}

\section{The abelian projective quotient}\label{sec:abelian}

The storage construction is unnecessary when the effective quotient is
abelian.  In that case the representation has a symplectic normal form and
the conclusion improves from a near-Parseval frame to an orthonormal basis.
We first recall the model that will be transported to the original
representation space.

For $x,\omega\in\R^d$, define the Weyl operator
\begin{equation}\label{eq:Weyl-operator}
 (W(x,\omega)F)(t)=\e^{2\pi i\omega\cdot t}F(t-x),
 \qquad F\in L^2(\R^d).
\end{equation}
If $u=(x,\omega)$ and $v=(x',\omega')$, direct multiplication gives
\begin{align}
 W(u)W(v)
 &=m_W(u,v)W(u+v),
 &m_W(u,v)&=\e^{-2\pi i\omega'\cdot x},
                                                        \label{eq:Weyl-multiplier}\\
 W(u)W(v)W(u)^{-1}W(v)^{-1}
 &=\e^{2\pi i\sigma(u,v)}I,
 &\sigma(u,v)&=\omega\cdot x'-\omega'\cdot x.
                                                        \label{eq:Weyl-commutator}
\end{align}
Thus the Weyl commutator is the exponential of the standard symplectic form.
The next lemma shows that every faithful projectively irreducible action of a
real vector group has, after a linear change of variables, exactly this
projective form.

\begin{lemma}[Faithful projective representations of vector groups%
]
\label{lem:projective-vector-group}
Let $\cH$ be an infinite-dimensional separable Hilbert space, let
$V=\R^r$, and let
\[
 U:V\longrightarrow\PU(\cH)
\]
be faithful, strongly continuous, and projectively irreducible.  Then
$r=2d$ for an integer $d\geq1$.  Moreover, there exist a linear isomorphism
$L:\R^{2d}\to V$ and a unitary $S:\cH\to L^2(\R^d)$ such that
\begin{equation}\label{eq:projective-Weyl-equivalence}
 (\operatorname{Ad}_S\circ U\circ L)(x,\omega)
 =[W(x,\omega)]
 \quad\text{in }\PU(L^2(\R^d)),
\end{equation}
where $\operatorname{Ad}_S([T])=[STS^{-1}]$.
\end{lemma}

\begin{proof}
We divide the proof into seven steps.  The first two steps lift the projective
action and record its multiplier and scalar commutator.  Steps~3--5 convert
that commutator into a nondegenerate symplectic form.  Step~6 matches the full
multiplier, rather than only its commutator, and Step~7 applies
Stone--von Neumann uniqueness.

\Step{Step 1: choose a continuous unitary lift}
The quotient map
\[
 q:\U(\cH)\longrightarrow\PU(\cH)
\]
is a locally trivial principal $\T$-bundle in the strong topology
\cite{Simms1970}.  Pulling this bundle back along $U$ gives a numerable
principal circle bundle over the paracompact contractible space $V=\R^r$.
The standard triviality theorem for principal bundles over contractible
paracompact bases implies that this pullback is trivial; see, for example,
\cite{Husemoller1994}.  Consequently, there is a strongly continuous map
\[
 \widetilde U:V\longrightarrow\U(\cH),
 \qquad q(\widetilde U(x))=U(x).
\]
After multiplying the lift by a constant scalar, we may assume
$\widetilde U(0)=I_{\cH}$.

\Step{Step 2: record the multiplier and scalar commutator}
Having chosen the lift, we record its failure to be a genuine
representation.  Because $U$ is a homomorphism, there is a unique scalar
$m(x,y)\in\T$ such
that
\begin{equation}\label{eq:lift-multiplier}
 \widetilde U(x)\widetilde U(y)=m(x,y)\widetilde U(x+y).
\end{equation}
The scalar depends continuously on $(x,y)$.  Indeed,
$\widetilde U(x)\widetilde U(y)\widetilde U(x+y)^{-1}$ is a strongly
continuous scalar operator, and its scalar is recovered from a matrix
coefficient against any fixed unit vector.  Associativity gives
\begin{equation}\label{eq:multiplier-identity}
 m(x,y)m(x+y,z)=m(y,z)m(x,y+z),
\end{equation}
and normalization gives $m(0,x)=m(x,0)=1$.

Since $V$ is abelian, the commutator of two representatives is scalar.  Set
\begin{equation}\label{eq:beta-definition}
 \beta(x,y)I_{\cH}
 =\widetilde U(x)\widetilde U(y)
  \widetilde U(x)^{-1}\widetilde U(y)^{-1}.
\end{equation}
Equivalently,
\begin{equation}\label{eq:beta-from-m}
 \beta(x,y)=m(x,y)m(y,x)^{-1}.
\end{equation}
Scalar phases cancel in \eqref{eq:beta-definition}, so $\beta$ is independent
of the chosen lift.

\Step{Step 3: prove that $\beta$ is an alternating bicharacter}
We next determine the algebraic structure carried by the commutator.
Continuity follows from \eqref{eq:beta-from-m}.  The commutators in
\eqref{eq:beta-definition} are central; hence the usual commutator identities
reduce to
\[
 \beta(x+x',y)=\beta(x,y)\beta(x',y),
 \qquad
 \beta(x,y+y')=\beta(x,y)\beta(x,y').
\]
Moreover,
\begin{equation}\label{eq:beta-alternating}
 \beta(x,x)=1,
 \qquad \beta(y,x)=\beta(x,y)^{-1}.
\end{equation}
Thus $\beta$ is a continuous alternating bicharacter.

\Step{Step 4: pass from the bicharacter to a bilinear form}
We now linearize the bicharacter.  Equip $\widehat V$ with the compact-open
topology.  The canonical map
\[
 V^*\longrightarrow\widehat V,
 \qquad
 \ell\longmapsto\bigl[x\mapsto\e^{2\pi i\ell(x)}\bigr],
\]
is a topological group isomorphism.  For each $y\in V$, the function
$x\mapsto\beta(x,y)$ is a continuous character, so there is a unique
$\ell_y\in V^*$ with $\beta(x,y)=\e^{2\pi i\ell_y(x)}$.  The joint
continuity of $\beta$ implies that $y\mapsto\beta(\,\cdot\,,y)$ is
continuous as a map into $\widehat V$ with its compact-open topology, and
hence that $y\mapsto\ell_y$ is continuous.  Bicharacter
multiplicativity makes this map additive, and every continuous additive map
between finite-dimensional real vector spaces is real-linear.  Consequently,
\begin{equation}\label{eq:B-definition}
 B(x,y):=\ell_y(x)
\end{equation}
is real bilinear and
\begin{equation}\label{eq:beta-B}
 \beta(x,y)=\e^{2\pi iB(x,y)}.
\end{equation}

Equation \eqref{eq:beta-alternating} gives $B(x,x)\in\mathbb Z$.  The
function $x\mapsto B(x,x)$ is continuous, $V$ is connected, and its value at
zero is zero.  Therefore $B(x,x)=0$ for all $x$.  Polarization over $\R$
then gives $B(x,y)=-B(y,x)$.

\Step{Step 5: use faithfulness to remove the radical}
The bilinear form is alternating, but it remains to prove that it is
nondegenerate.  Let
\[
 \operatorname{rad}(B)
 =\{x\in V:B(x,y)=0\text{ for every }y\in V\}.
\]
If $x\in\operatorname{rad}(B)$, then $\beta(x,y)=1$ for every $y$, so
$\widetilde U(x)$ commutes with all $\widetilde U(y)$.  Projective
irreducibility and Schur's lemma imply that $\widetilde U(x)$ is scalar.
Thus $U(x)$ is the identity in $\PU(\cH)$; faithfulness gives $x=0$.
Hence $B$ is nondegenerate.

A nondegenerate alternating form exists only in even dimension.  Thus
$r=2d$, and elementary symplectic linear algebra provides a linear
isomorphism $L:\R^{2d}\to V$ such that
\begin{equation}\label{eq:Darboux-normalization}
 B\bigl(L(x,\omega),L(x',\omega')\bigr)
 =\omega\cdot x'-\omega'\cdot x.
\end{equation}
The integer $d$ is positive: if $V=\{0\}$, projective irreducibility would
force $\dim\cH=1$.

\Step{Step 6: match the multiplier, not merely its commutator}
The commutator determines the symplectic geometry, but projective unitary
equivalence requires control of the entire multiplier.  Let $m_L$ be the
multiplier of $u\mapsto\widetilde U(Lu)$, and let $m_W$ be
the multiplier in \eqref{eq:Weyl-multiplier}.  Equations
\eqref{eq:beta-B}, \eqref{eq:Darboux-normalization}, and
\eqref{eq:Weyl-commutator} show that the two multipliers have the same
commutator.  Therefore
\[
 c(u,v):=m_L(u,v)m_W(u,v)^{-1}
\]
is a continuous symmetric multiplier on $\R^{2d}$.

We verify directly that $c$ is a continuous coboundary.  Give
$E_c=\T\times\R^{2d}$ the product topology and the multiplication
\begin{equation}\label{eq:central-extension-product}
 (z,u)(w,v)=(zwc(u,v),u+v).
\end{equation}
The multiplier identity makes this operation associative, and symmetry of
$c$ makes $E_c$ abelian.  Its identity is $(1,0)$ and
\[
 (z,u)^{-1}=\bigl(z^{-1}c(u,-u)^{-1},-u\bigr).
\]
Continuity of $c$ makes multiplication and inversion continuous.  Thus
$E_c$ is a locally compact abelian group, and $\T\times\{0\}$ is a closed
subgroup.  By the character-extension theorem
\cite[Corollary~4.42]{FollandAbstract2016}, the character $(z,0)\mapsto z$
extends to a continuous character $\chi$ of $E_c$.  Set $a(u)=\chi(1,u)$.
Applying $\chi$ to $(1,u)(1,v)=(c(u,v),u+v)$ gives
\begin{equation}\label{eq:c-coboundary}
 c(u,v)=\frac{a(u)a(v)}{a(u+v)}.
\end{equation}
The rephased lift
\begin{equation}\label{eq:rephased-lift}
 \widehat U(u)=a(u)^{-1}\widetilde U(Lu)
\end{equation}
therefore has multiplier exactly $m_W$.  This is the explicit continuous
cohomological step behind the general framework in \cite{Kleppner1965}.

\Step{Step 7: apply Stone--von Neumann uniqueness}
Having matched the multipliers, we can invoke uniqueness.  The map
$\widehat U$ is strongly continuous, irreducible, and has exactly the
multiplier $m_W$.  The Stone--von Neumann uniqueness theorem in multiplier
form states that every irreducible strongly continuous $m_W$-representation
of $\R^{2d}$ is unitarily equivalent to the Schr\"odinger--Weyl
representation with the same multiplier.  Hence there is a unitary
$S:\cH\to L^2(\R^d)$ such that
\[
 S\widehat U(x,\omega)S^{-1}=W(x,\omega).
\]
Passing to projective classes removes the rephasing in
\eqref{eq:rephased-lift} and proves
\eqref{eq:projective-Weyl-equivalence}.  See
\cite[Theorem~3.3 and p.~314]{BaggettKleppner1973} for multiplier
representations and \cite[Theorem~1.50]{Folland1989} for the classical
Stone--von Neumann theorem.  This completes the proof.
\end{proof}

The faithfulness assumption in the preceding lemma is essential.  The
following elementary example shows that its precise role is to remove the
radical of the commutator form before the Weyl model is invoked.

\begin{example}[Why the effective quotient is necessary]
\label{ex:radical-before-quotient}
On $L^2(\R)$ define a projective representation of $\R^3$ by
\[
 U(x,\omega,s)=[W(x,\omega)].
\]
The third coordinate acts trivially in $\PU(L^2(\R))$.  The commutator form
\[
 B\bigl((x,\omega,s),(x',\omega',s')\bigr)
 =\omega x'-\omega'x
\]
has radical $\{(0,0,s):s\in\R\}$.  Passing to the effective quotient removes
exactly this radical and leaves the symplectic plane $\R^2$.  This example
pinpoints the use of faithfulness in Step~5.
\end{example}

We can now transport the standard Gabor orthonormal basis through the
projective equivalence of Lemma~\ref{lem:projective-vector-group}.  This
gives the complete conclusion in the abelian branch.

\begin{proposition}[The abelian branch]
\label{prop:abelian-quotient}
Let $G$ be an exponential solvable Lie group, let
$\pi:G\to\U(\cH)$ be an infinite-dimensional irreducible unitary
representation, and put $K=\pker(\pi)$.  If $Q=G/K$ is abelian, then there
exist a unit vector $\phi\in\cH$ and a countable set $\Gamma\subset G$ such
that
\[
 \{\pi(\gamma)\phi:\gamma\in\Gamma\}
\]
is an orthonormal basis of $\cH$.  The image of $\Gamma$ in $Q$ is a closed
discrete subgroup.
\end{proposition}

\begin{proof}
The proof proceeds by identifying the effective quotient, placing its action
in Weyl form, and then lifting the standard lattice back to $G$.

\Step{Step 1: identify the quotient with a vector group}
By Theorem~\ref{thm:projective-c0-input}, $Q$ is exponential.  Since it is
also abelian, its exponential map is a Lie-group isomorphism from the
additive Lie algebra onto $Q$.  Hence $Q\simeq\R^r$.  This identification
places the effective quotient within the scope of
Lemma~\ref{lem:projective-vector-group}.

\Step{Step 2: use the effective projective representation}
With this identification in place,
Proposition~\ref{prop:quotient-dictionary} gives a faithful, strongly
continuous, projectively irreducible homomorphism
\[
 U:Q\longrightarrow\PU(\cH),
 \qquad U(gK)=[\pi(g)].
\]
Thus faithfulness, strong continuity, and projective irreducibility---the
three hypotheses of Lemma~\ref{lem:projective-vector-group}---are all
available.

\Step{Step 3: put the quotient action in Weyl form}
We may now apply Lemma~\ref{lem:projective-vector-group}.  It gives $r=2d$,
a linear isomorphism
$L:\R^{2d}\to Q$, and a unitary $S:\cH\to L^2(\R^d)$ satisfying
\eqref{eq:projective-Weyl-equivalence}.  The orbit problem has therefore been
reduced to the explicit Weyl model, where a canonical lattice orthonormal
basis is available.

\Step{Step 4: verify the model orthonormal basis}
We next identify the orthonormal basis in the model representation.  Let
$C=[0,1)^d$ and $\phi_0=\mathbf1_C$.  The family
\begin{equation}\label{eq:Gabor-ONB}
 \mathcal B_0
 =\{W(k,\ell)\phi_0:(k,\ell)\in\mathbb Z^d\times\mathbb Z^d\}
\end{equation}
is an orthonormal basis of $L^2(\R^d)$.  If $k\neq k'$, the corresponding
functions have disjoint supports $k+C$ and $k'+C$ up to null sets.  For a
fixed $k$, translation identifies the functions indexed by
$\ell\in\mathbb Z^d$ with the standard Fourier basis of $L^2(C)$.  Finally,
\[
 L^2(\R^d)=\bigoplus_{k\in\mathbb Z^d}L^2(k+C),
\]
which proves completeness and hence establishes the model assertion.

\Step{Step 5: transport and lift the lattice}
It remains to transport the model family and lift its parameters.  The set
\begin{equation}\label{eq:quotient-lattice}
 \Lambda=L(\mathbb Z^{2d})\subset Q
\end{equation}
is a closed discrete subgroup.  Put $\phi=S^{-1}\phi_0$.  For each
$\lambda=L(k,\ell)\in\Lambda$, choose $g_\lambda\in G$ with
$g_\lambda K=\lambda$.  Equation
\eqref{eq:projective-Weyl-equivalence} gives a scalar $z_\lambda\in\T$ for
which
\begin{equation}\label{eq:lifted-Weyl-atom}
 S\pi(g_\lambda)\phi=z_\lambda W(k,\ell)\phi_0.
\end{equation}
Unimodular scalars preserve orthonormality and completeness.  Thus
\[
 \Gamma=\{g_\lambda:\lambda\in\Lambda\}
\]
has the asserted properties, and its image in $Q$ is exactly $\Lambda$.
This completes the proof.
\end{proof}

The next two examples make the transport procedure explicit.  The first
recovers the standard Gabor basis from the Schr\"odinger representation,
whereas the second shows how a nonstandard lattice arises in the original
quotient coordinates.

\begin{example}[The Schr\"odinger representation]\label{ex:heisenberg}
Let $H_d=\R^d\times\R^d\times\R$ with multiplication
\[
 (x,\omega,z)(x',\omega',z')
 =\bigl(x+x',\omega+\omega',z+z'-\omega'\cdot x\bigr).
\]
For $\lambda\neq0$, define
\begin{equation}\label{eq:Schrodinger-representation}
 \bigl(\pi_\lambda(x,\omega,z)F\bigr)(t)
 =\e^{2\pi i\lambda(z+\omega\cdot t)}F(t-x).
\end{equation}
The projective kernel is the center
$K=\{(0,0,z):z\in\R\}$, and $H_d/K\simeq\R^{2d}$.  With
$T_xF(t)=F(t-x)$ and $M_\eta F(t)=\e^{2\pi i\eta\cdot t}F(t)$, the identity
$T_xM_\eta=\e^{-2\pi i\eta\cdot x}M_\eta T_x$ verifies the displayed
group law.  Moreover, $M_{\lambda\omega}T_x$ is scalar only when
$x=\omega=0$, which verifies the projective-kernel assertion.  With
$\phi_0=\mathbf1_{[0,1)^d}$ and
\[
 \Gamma_\lambda
 =\{(k,\lambda^{-1}\ell,0):k,\ell\in\mathbb Z^d\},
\]
one has
\[
 \pi_\lambda(k,\lambda^{-1}\ell,0)\phi_0
 =W(k,\ell)\phi_0.
\]
The abstract lattice construction therefore recovers the usual Gabor
orthonormal basis exactly.
\end{example}

\begin{example}[An anisotropic phase-space lattice]
\label{ex:anisotropic}
Let $A\in\operatorname{GL}_d(\R)$ and give
$G_A=\R^d\times\R^d\times\R$ the multiplication
\[
 (x,\omega,z)(x',\omega',z')
 =\bigl(x+x',\omega+\omega',
        z+z'-(A\omega')\cdot x\bigr).
\]
The irreducible representation
\[
 \bigl(\pi_A(x,\omega,z)F\bigr)(t)
 =\e^{2\pi iz}\e^{2\pi i(A\omega)\cdot t}F(t-x)
\]
has the central $z$-axis as projective kernel.  On the quotient its
commutator form is
\[
 B_A\bigl((x,\omega),(x',\omega')\bigr)
 =(A\omega)\cdot x'-(A\omega')\cdot x.
\]
Indeed, the linear change of variables
$(x,\omega,z)\mapsto(x,A\omega,z)$ identifies this group and representation
with the standard Heisenberg--Weyl model; irreducibility and the stated
projective kernel follow at once.  Under this identification, the Darboux
map is $L_A(x,\eta)=(x,A^{-1}\eta)$.  Hence
\[
 \Gamma_A
 =\{(k,A^{-1}\ell,0):k,\ell\in\mathbb Z^d\}
\]
and $\mathbf1_{[0,1)^d}$ generate an orbit orthonormal basis.  This example
shows why the lattice in the original quotient coordinates need not be the
standard integer lattice.
\end{example}

\begin{remark}[Strength of the abelian branch]
Proposition~\ref{prop:abelian-quotient} is stronger than the abelian case of
Theorem~\ref{thm:main}: an orthonormal basis satisfies
\eqref{eq:near-parseval-frame} for every $0<\varepsilon<1$.  Approximation enters only
in the nonabelian storage branch.
\end{remark}

\section{Irreducible orbit frames and final remarks}
\label{sec:proof-main}

We now combine the projective-$C_0$ input with the two mutually exclusive
geometries of the effective quotient.  This final dichotomy proves the main
theorem and also clarifies exactly where approximation is needed.

\begin{proof}[Proof of Theorem~\ref{thm:main}]
Set $K=\pker(\pi)$ and let $q:G\to Q=G/K$ be the quotient map.  We first
collect the structural information proved earlier.  By
Theorem~\ref{thm:projective-c0-input}, the subgroup $K$ is connected and the
quotient $Q$ is exponential.  Moreover,
Corollary~\ref{cor:cross-c0} shows that $\pi$ is projectively $C_0$.

\Step{Step 1: the nonabelian effective quotient}
Suppose first that $Q$ is nonabelian.  In this case,
Corollary~\ref{cor:lifted-escape} supplies a
local mixing-expansion datum for $\pi$.  Theorem~\ref{thm:abstract-storage}
then produces a nonzero vector
$\phi\in\cH$ and a countable set $\Gamma\subset G$ satisfying
\eqref{eq:near-parseval-frame}.  It also makes $q$ injective on $\Gamma$ and
$q(\Gamma)$ left relatively separated.

\Step{Step 2: the abelian effective quotient}
We now consider the remaining case and suppose that $Q$ is abelian.
Proposition~\ref{prop:abelian-quotient}
produces a unit vector $\phi$ and one representative in $G$ of each point of
a lattice $\Lambda\subset Q$ such that the resulting orbit is an
orthonormal basis.  Thus $q$ is injective on the chosen set $\Gamma$,
$q(\Gamma)=\Lambda$ is relatively separated, and an orthonormal basis
satisfies \eqref{eq:near-parseval-frame} for every
$0<\varepsilon<1$.

The two cases exhaust all possibilities for $Q$, and the theorem follows.
\end{proof}

\begin{remark}[Why near-Parseval does not automatically mean Parseval]
In the nonabelian branch, the construction gives bounds
$(1-\varepsilon)^2$ and $(1+\varepsilon)^2$ for every $\varepsilon>0$.
This does not produce an exact Parseval orbit frame.  If $S_\Gamma$ is the
frame operator, canonical tightening gives the Parseval family
\[
 \bigl(S_\Gamma^{-1/2}\pi(\gamma)\phi\bigr)_{\gamma\in\Gamma}.
\]
For this family to remain the orbit of one vector, one would need enough
commutation between $S_\Gamma$ and the sampled operators.  Such commutation
is not supplied by the storage construction.  Exact Parseval or
orthonormal refinements therefore remain a separate question when the
effective quotient is nonabelian.
\end{remark}

\begin{remark}[Where exponentiality remains used]
Exponentiality is a sufficient hypothesis for Theorem~\ref{thm:main}, not a
hypothesis of the storage theorem.  In the exponential solvable application,
published results supply a connected projective kernel, an exponential
effective quotient, and projective coefficient decay.  The spectral criterion
then excludes recurrent nonzero purely imaginary adjoint eigenvalues and
forces exponential or polynomial conjugation escape.  A general connected
solvable Lie group may have recurrent adjoint rotations, so no universal
nonexponential-solvable statement is asserted.  Nevertheless, any individual
representation in that larger class is covered as soon as it admits the local
mixing-expansion datum of Definition~\ref{def:mixing-expansion}.  The
semisimple and local-field theorems demonstrate that exponentiality and even
solvability are not intrinsic to the storage mechanism.
\end{remark}

\begin{remark}[What the examples show]
Examples~\ref{ex:affine-escape} and \ref{ex:nilpotent-escape} exhibit the two
growth mechanisms behind the nonabelian branch.  Examples
\ref{ex:heisenberg} and \ref{ex:anisotropic} show how an abelian effective
quotient recovers concrete Gabor orthonormal bases, possibly on a
nonstandard lattice.  Example~\ref{ex:finite-parseval-block} explains the
redundant Parseval target, while Example~\ref{ex:explicit-weights} explains
the rapid scale separation required to store infinitely many blocks in one
vector.
\end{remark}

\section{A nonnilpotent triangular model in dimension four}
\label{sec:B4-model}

We conclude the body of the paper by making the preceding mechanism fully
explicit for the connected group
\[
 B_4^+
 =
 \left\{
 g=(g_{ij})_{i,j=1}^4:
 g_{ij}=0\ \text{if }i>j,\quad g_{ii}>0
 \right\}.
\]
This choice displays two features at once: the group $B_4^+$ is completely
solvable and nonnilpotent, and, unlike a unitriangular group, it retains a
nontrivial diagonal action while remaining within the exponential
class.

\subsection{Scalar, diagonal, and unitriangular factors}

Let
\[
 A=\{\operatorname{diag}(a_1,a_2,a_3,a_4):a_j>0\},
 \qquad
 Z_+=\{rI_4:r>0\},
\]
and let $N=UT_4(\R)$ be the unitriangular subgroup.  With $A$ and $N$ so
defined, every element of $B_4^+$ admits a unique factorization
\[
 g=na,\qquad n\in N,\quad a\in A.
\]
We write
\[
 n(x,y,z)=
 \begin{pmatrix}
 1&x_1&y_1&z\\
 0&1&x_2&y_2\\
 0&0&1&x_3\\
 0&0&0&1
 \end{pmatrix}.
\]
The multiplication in the nilpotent factor will be used repeatedly below.
Direct matrix multiplication gives
\begin{align}
 &n(x,y,z)n(x',y',z') \notag\\
 &\quad =
 n\Bigl(
 x+x',
 \bigl(y_1+y_1'+x_1x_2',\,y_2+y_2'+x_2x_3'\bigr),
 z+z'+x_1y_2'+y_1x_3'
 \Bigr).
 \label{eq:B4-N-law}
\end{align}
The commutator subgroup is now visible directly from this formula.  Namely,
\begin{equation}
 [N,N]
 =N_2
 =\left\{
 \begin{pmatrix}
 1&0&y_1&z\\
 0&1&0&y_2\\
 0&0&1&0\\
 0&0&0&1
 \end{pmatrix}
 :y_1,y_2,z\in\R
 \right\}.
 \label{eq:B4-N2}
\end{equation}
Accordingly, $N/N_2\simeq\R^3$, and we use the quotient coordinates
\[
 x(n)=(n_{12},n_{23},n_{34}).
\]

We next separate the scalar direction from the effective diagonal action.
The determinant-one factor is
\[
 B_4^{+,1}=\{g\in B_4^+:\det g=1\}.
\]
The center of $B_4^+$ is $Z_+$, and multiplication gives a direct-product
isomorphism
\begin{equation}
 \R_{>0}\times B_4^{+,1}\longrightarrow B_4^+,
 \qquad
 (r,h)\longmapsto rh.
 \label{eq:B4-scalar-splitting}
\end{equation}
The inverse is equally explicit: it sends $g$ to
\[
 \left((\det g)^{1/4},\,(\det g)^{-1/4}g\right).
\]
Because the only positive scalar matrix of determinant one is $I_4$, the
group $B_4^{+,1}$ has trivial center.

To express the remaining diagonal action in coordinates adapted to the
simple roots, let $a=\operatorname{diag}(a_1,a_2,a_3,a_4)$ and put
\begin{equation}
 s_j(a)=\log\frac{a_{j+1}}{a_j},
 \qquad 1\leq j\leq3.
 \label{eq:B4-simple-root-coordinates}
\end{equation}
Then $aZ_+\mapsto s(a)$ identifies $A/Z_+$ with $\R^3$, and its
determinant-one inverse is
\begin{equation}
 a(t)=\operatorname{diag}\left(
 \e^{-(3t_1+2t_2+t_3)/4},
 \e^{(t_1-2t_2-t_3)/4},
 \e^{(t_1+2t_2-t_3)/4},
 \e^{(t_1+2t_2+3t_3)/4}
 \right).
 \label{eq:B4-diagonal-section}
\end{equation}

These coordinates are particularly convenient for conjugation.  Indeed,
\[
 aE_{ij}a^{-1}=\frac{a_i}{a_j}E_{ij}.
\]
Therefore, modulo $N_2$,
\[
 x_j(ana^{-1})=\e^{-s_j(a)}x_j(n).
\]
Combining this identity with the preceding factorization, we obtain, for
$g=na$ and $g'=n'a'$,
\begin{equation}
 x(gg')=x(n)+\e^{-s(a)}\odot x(n'),
 \qquad
 s(aa')=s(a)+s(a'),
 \label{eq:B4-effective-law}
\end{equation}
where $\odot$ denotes coordinatewise multiplication.

\subsection{Complete solvability and failure of nilpotence}

The same triangular structure also makes the solvability properties
transparent.  Let $\mathfrak b_4$ be the Lie algebra of all real
upper-triangular $4\times4$ matrices.  For
$D=\operatorname{diag}(d_1,d_2,d_3,d_4)$,
\[
 [D,E_{ij}]=(d_i-d_j)E_{ij}.
\]
In a basis ordered by root height, every $\ad X$, $X\in\mathfrak b_4$,
is upper triangular over $\R$, with diagonal entries among
\[
 0,\qquad d_i-d_j\quad (i<j).
\]
Thus every adjoint map has only real eigenvalues, and $\mathfrak b_4$ is
completely solvable.

Complete solvability should not be confused here with nilpotence.  Indeed,
with
\[
 D_1=\frac14\operatorname{diag}(3,-1,-1,-1)
\]
one has
\[
 [D_1,E_{12}]=E_{12}.
\]
Thus $\ad D_1$ is not nilpotent, and neither $\mathfrak b_4$ nor
$B_4^+$ is nilpotent.  At the same time, $B_4^+\simeq\R^{10}$ is connected
and simply connected, and its adjoint spectra contain no nonzero purely
imaginary numbers.  The standard exponentiality criterion therefore shows
that the exponential map is a global diffeomorphism:
\[
 \exp:\mathfrak b_4\longrightarrow B_4^+.
\]

\begin{remark}
Although this example is nonnilpotent, it does not lie beyond the class of
exponential solvable Lie groups.  Its role is instead to show concretely that
the storage theorem reaches well beyond the nilpotent setting.
\end{remark}

\subsection{An explicit irreducible representation}

Having isolated the effective coordinates, we now construct a representation
in which all three simple-root directions remain visible.  Fix
\[
 \xi=(\xi_1,\xi_2,\xi_3)\in(\R\setminus\{0\})^3.
\]
The formula
\begin{equation}
 \chi_\xi(n)
 =\exp\left(2\pi i\sum_{j=1}^3\xi_jx_j(n)\right)
 \label{eq:B4-N-character}
\end{equation}
defines a unitary character of $N$, since the first-superdiagonal coordinates
add under multiplication.  Its behavior under diagonal conjugation is
\[
 \chi_\xi(a^{-1}na)
 =\exp\left(
 2\pi i\sum_{j=1}^3\xi_j\e^{s_j(a)}x_j(n)
 \right).
\]
Because every $\xi_j$ is nonzero, the stabilizer of $\chi_\xi$ in $A$ is
exactly $Z_+$.  Moreover, inner automorphisms of $N$ act trivially on
$N/[N,N]$; hence the stabilizer in $B_4^+$ is
\[
 H=NZ_+.
\]

The induced realization associated with $\chi_\xi$ can now be written
explicitly.
For $a\in A$, set
\[
 r(a)=(\det a)^{1/4}.
\]
For $\tau\in\R$, define
$\pi_{\xi,\tau}:B_4^+\to\U(L^2(\R^3))$ by
\begin{equation}
 \bigl(\pi_{\xi,\tau}(na)F\bigr)(t)
 =r(a)^{i\tau}
 \exp\left(
 2\pi i\sum_{j=1}^3\xi_j\e^{t_j}x_j(n)
 \right)
 F\bigl(t-s(a)\bigr).
 \label{eq:B4-explicit-representation}
\end{equation}
Here $r^{i\tau}=\e^{i\tau\log r}$.

\begin{proposition}
\label{prop:B4-explicit-representation}
For every $\xi\in(\R\setminus\{0\})^3$ and $\tau\in\R$, the formula
\eqref{eq:B4-explicit-representation} defines an infinite-dimensional
irreducible strongly continuous unitary representation of $B_4^+$.
\end{proposition}

\begin{proof}
We verify, in order, unitarity, multiplicativity, strong continuity, and
irreducibility.

The formula is unitary because the first two factors in
\eqref{eq:B4-explicit-representation} are unimodular, whereas the translation
$F(t)\mapsto F(t-s(a))$ preserves Lebesgue measure.

For the representation law, write $u=x(n)$, $u'=x(n')$,
$s=s(a)$, and $s'=s(a')$.  Equations
\eqref{eq:B4-effective-law} and
\[
 \e^{t-s}\odot u'=\e^t\odot(\e^{-s}\odot u')
\]
show that the product of the two phase factors is precisely the phase
associated with
\[
 (u,s)(u',s')
 =\bigl(u+\e^{-s}\odot u',\,s+s'\bigr).
\]
Since also $r(aa')=r(a)r(a')$, multiplicativity follows.  Strong continuity
is first obtained on compactly supported continuous functions and then, by
density, on all of $L^2(\R^3)$.

We finally turn to irreducibility.  The operators corresponding to $n\in N$
include all multiplication operators
\[
 M_uF(t)=
 \exp\left(2\pi i\sum_{j=1}^3\xi_j\e^{t_j}u_j\right)F(t),
 \qquad u\in\R^3.
\]
Because
\[
 t\longmapsto(\xi_1\e^{t_1},\xi_2\e^{t_2},\xi_3\e^{t_3})
\]
is a Borel isomorphism from $\R^3$ onto an open orthant, these characters
separate points of the open orthant and generate its Borel $\sigma$-algebra.
Pulling them back through the displayed Borel isomorphism shows that the von
Neumann algebra generated by the operators $M_u$ is the full multiplication
algebra $L^\infty(\R^3)$.  Consequently, an operator commuting with all
$M_u$ must be multiplication by some
$h\in L^\infty(\R^3)$.  The diagonal subgroup, in turn, supplies every
translation $F(t)\mapsto F(t-s)$, $s\in\R^3$.  Commutation with these
translations forces $h(t)=h(t-s)$ almost everywhere for every $s$, so that
$h$ is essentially constant.  The commutant is therefore $\mathbb CI$,
which proves irreducibility.
\end{proof}

\subsection{The exact projective kernel and effective quotient}

The representation has now been identified, and we next determine the
parameter space on which its orbit vectors are genuinely distinct.  The
following proposition computes both the projective kernel and the resulting
effective group.

\begin{proposition}
\label{prop:B4-projective-kernel}
The projective kernel of $\pi_{\xi,\tau}$ is
\begin{equation}
 K=\pker(\pi_{\xi,\tau})=Z_+N_2.
 \label{eq:B4-projective-kernel}
\end{equation}
Consequently,
\begin{align}
 Q=B_4^+/K
 &\simeq B_4^{+,1}/N_2 \notag\\
 &\simeq\prod_{j=1}^3\bigl(\R\rtimes\R\bigr),
 \label{eq:B4-effective-quotient}
\end{align}
where each factor has multiplication
\begin{equation}
 (u_j,s_j)(u_j',s_j')
 =\bigl(u_j+\e^{-s_j}u_j',\,s_j+s_j'\bigr).
 \label{eq:B4-affine-law}
\end{equation}
In particular, $Q$ is a six-dimensional nonabelian exponential Lie group.
\end{proposition}

\begin{proof}
Every element of $N_2$ has $x(n)=0$ and therefore acts trivially, whereas
every $rI_4\in Z_+$ acts by the scalar $r^{i\tau}$.  Hence
$Z_+N_2\subseteq\pker(\pi_{\xi,\tau})$.

For the reverse inclusion, suppose that $\pi_{\xi,\tau}(na)$ is scalar.
A nonzero translation cannot be hidden by the phase multiplier in
\eqref{eq:B4-explicit-representation}: if $s(a)\neq0$, choose a small ball
$E\subset\R^3$ whose closure is disjoint from $E+s(a)$.  If a nonzero $F$
is supported in $E$, then $\pi_{\xi,\tau}(na)F$ is supported in $E+s(a)$,
whereas every scalar multiple of $F$ is supported in $E$.  It follows that
$s(a)=0$, and hence that $a$ is a positive scalar matrix.  The remaining
multiplication function
\[
 t\longmapsto
 \exp\left(2\pi i\sum_{j=1}^3\xi_j\e^{t_j}x_j(n)\right)
\]
is constant only if $x_j(n)=0$ for all $j$, since every $\xi_j$ is nonzero.
Thus $n\in N_2$, which proves the reverse inclusion.  The quotient law is
exactly \eqref{eq:B4-effective-law}, coordinate by coordinate.
\end{proof}

Proposition~\ref{prop:B4-projective-kernel} therefore separates the geometry
cleanly: the higher-root directions $E_{13},E_{24},E_{14}$ are projectively
invisible in this family, while the three simple-root directions
$E_{12},E_{23},E_{34}$ survive.  In particular, the effective group is not
abelian; it is the product of three independent affine geometries.

\subsection{Direct coefficient decay}

The preceding coordinates also make projective $C_0$ decay directly
accessible, without any appeal to the general representation-theoretic
input.

\begin{proposition}
\label{prop:B4-direct-C0}
For all $F,G\in L^2(\R^3)$, the coefficient modulus
\[
 gK\longmapsto
 \left|\langle F,\pi_{\xi,\tau}(g)G\rangle\right|
\]
belongs to $C_0(Q)$.
\end{proposition}

\begin{proof}
By density, it suffices first to take $F,G\in C_c^\infty(\R^3)$.  In the
coordinates $(u,s)$ of \eqref{eq:B4-effective-quotient}, and with the inner
product linear in its first variable, the coefficient is
\begin{equation}
 c_{F,G}(u,s)
 =\int_{\R^3}
 F(t)\overline{G(t-s)}
 \exp\left(
 -2\pi i\sum_{j=1}^3\xi_ju_j\e^{t_j}
 \right)\,dt.
 \label{eq:B4-coefficient}
\end{equation}
When $s$ leaves every compact set, the two compact supports are eventually
disjoint, and the coefficient is therefore zero.  It remains to consider
$s$ in a compact set while $|u|\to\infty$.  After the change of variables
$y_j=\e^{t_j}$, the integral becomes the Euclidean Fourier transform, at
frequency
$(\xi_1u_1,\xi_2u_2,\xi_3u_3)$, of an $L^1$-function supported in a compact
subset of $(0,\infty)^3$.  As $s$ ranges over a fixed compact set, the
resulting amplitudes depend continuously on $s$ in $L^1$; their range is
therefore compact in $L^1$.  The Riemann--Lebesgue lemma, applied uniformly
on compact subsets of $L^1$, therefore gives decay uniformly in $s$.  Since every $\xi_j$
is nonzero, $|u|\to\infty$ if and only if the resulting frequency leaves
every compact set.  This proves the assertion on the dense subspace.  To
pass to general $F,G\in L^2(\R^3)$, choose
$F_0,G_0\in C_c^\infty(\R^3)$.  Unitarity gives, uniformly in
$g\in B_4^+$,
\[
 \left|
 \langle F,\pi_{\xi,\tau}(g)G\rangle-
 \langle F_0,\pi_{\xi,\tau}(g)G_0\rangle
 \right|
 \leq
 \|F-F_0\|_2\|G\|_2+
 \|F_0\|_2\|G-G_0\|_2.
\]
The coefficient functions, and hence also their moduli, are therefore
uniform limits of $C_0(Q)$ functions.  The conclusion follows for arbitrary
$F,G\in L^2(\R^3)$.
\end{proof}

\subsection{Exact storage and retrieval directions}

It remains to exhibit the geometric datum that drives the storage
construction.  Set
\[
 \begin{aligned}
 D_1&=\frac14\operatorname{diag}(3,-1,-1,-1),\\
 b(R)&=\exp(RD_1)\\
 &=\operatorname{diag}
 \left(\e^{3R/4},\e^{-R/4},\e^{-R/4},\e^{-R/4}\right),
 \end{aligned}
\]
and
\[
 p(\theta)=\exp(\theta E_{12})=I_4+\theta E_{12}.
\]
Both curves lie in $B_4^{+,1}$.  Since
\[
 [D_1,E_{12}]=E_{12},
\]
the exact conjugation formula is
\begin{align}
 b(R)p(\theta)^{-1}p(\theta')b(R)^{-1}
 &=b(R)\bigl(I_4+(\theta'-\theta)E_{12}\bigr)b(R)^{-1}\notag\\
 &=I_4+\e^R(\theta'-\theta)E_{12}.
 \label{eq:B4-exact-escape}
\end{align}
Let $q:B_4^+\to Q=B_4^+/K$ denote the quotient map.  In the quotient
coordinates introduced above,
\[
 q(b(R))=(0,(-R,0,0))
\]
and
\[
 q\!\left(
 b(R)p(\theta)^{-1}p(\theta')b(R)^{-1}
 \right)
 =\bigl(\e^R(\theta'-\theta)e_1,0\bigr).
\]
The displayed formula for $q(b(R))$ shows that $q(b(R))\to\infty$, while
for $\theta\neq\theta'$ the conjugation computation yields
\[
 q\!\left(
 b(R)p(\theta)^{-1}p(\theta')b(R)^{-1}
 \right)\longrightarrow\infty.
\]
Together, these are precisely the ordinary and conjugation escape hypotheses
of Corollary~\ref{cor:projective-storage}.

The same separation can also be read at the operator level:
\[
 \pi_{\xi,\tau}(b(R))
 \pi_{\xi,\tau}(p(\theta))^{-1}
 \pi_{\xi,\tau}(p(\theta'))
 \pi_{\xi,\tau}(b(R))^{-1}
 =M_{\exp(2\pi i\xi_1\e^{t_1+R}(\theta'-\theta))}.
\]
Hence the diagonal flow amplifies every nonzero difference between two probe
parameters at the exact exponential rate $\e^R$.

\begin{figure}[H]
\centering
\begin{minipage}[c]{0.53\linewidth}
\centering
\begin{tikzpicture}[
 vertex/.style={circle,draw=OrbitInk!55,fill=white,
   minimum size=7mm,font=\small\bfseries},
 keep/.style={-{Stealth[length=2.1mm]},OrbitBlue,line width=1.5pt},
 kill/.style={-{Stealth[length=1.8mm]},OrbitInk!60,
   densely dashed,line width=1.05pt}
]
 \node[vertex] (v1) at (0,0) {1};
 \node[vertex] (v2) at (1.8,0) {2};
 \node[vertex] (v3) at (3.6,0) {3};
 \node[vertex] (v4) at (5.4,0) {4};

 \draw[keep] (v1) -- node[below=3pt,font=\scriptsize] {$E_{12}$} (v2);
 \draw[keep] (v2) -- node[below=3pt,font=\scriptsize] {$E_{23}$} (v3);
 \draw[keep] (v3) -- node[below=3pt,font=\scriptsize] {$E_{34}$} (v4);

 \draw[kill,bend left=38]
   (v1) to node[above=3pt,font=\scriptsize] {$E_{13}$} (v3);
 \draw[kill,bend left=38]
   (v2) to node[above=3pt,font=\scriptsize] {$E_{24}$} (v4);
 \draw[kill,bend left=62]
   (v1) to node[above=20pt,font=\scriptsize] {$E_{14}$} (v4);

 \draw[keep] (0,-1.15) -- +(0.72,0)
   node[right=3pt,text=black,font=\scriptsize] {survives in $Q$};
 \draw[kill] (2.85,-1.15) -- +(0.72,0)
   node[right=3pt,text=black,font=\scriptsize] {lies in $N_2\subset K$};
\end{tikzpicture}
\end{minipage}
\hfill
\begin{minipage}[c]{0.43\linewidth}
\centering
\begin{tikzpicture}
\begin{semilogyaxis}[
 width=.97\linewidth,
 height=5.1cm,
 domain=0:6,
 samples=160,
 xmin=0,xmax=6,
 ymin=0.08,ymax=500,
 xlabel={$R$},
 ylabel={$|u_1(R)|$},
 title={Exact conjugation escape},
 title style={font=\small\bfseries},
 label style={font=\small},
 tick label style={font=\scriptsize},
 grid=both,
 major grid style={OrbitInk!14},
 minor grid style={OrbitInk!7},
 axis line style={OrbitInk!60},
 legend style={draw=none,fill=none,font=\scriptsize,
   at={(0.03,0.97)},anchor=north west}
]
 \addplot[OrbitBlue,very thick] {exp(x)};
 \addlegendentry{$|\theta'-\theta|=1$}
 \addplot[OrbitGold,very thick,dash pattern=on 7pt off 2pt]
   {0.3*exp(x)};
 \addlegendentry{$|\theta'-\theta|=0.3$}
 \addplot[OrbitInk!65,very thick,densely dashed]
   {0.1*exp(x)};
 \addlegendentry{$|\theta'-\theta|=0.1$}
\end{semilogyaxis}
\end{tikzpicture}
\end{minipage}
\caption{Left: the simple-root directions survive the projective quotient,
whereas the higher-root subgroup $N_2=[N,N]$ is projectively invisible.
Right: the surviving $E_{12}$ coordinate after conjugation is exactly
$\e^R|\theta'-\theta|$, so each pair of distinct retrieval parameters
separates exponentially.}
\label{fig:B4-effective-geometry}
\end{figure}
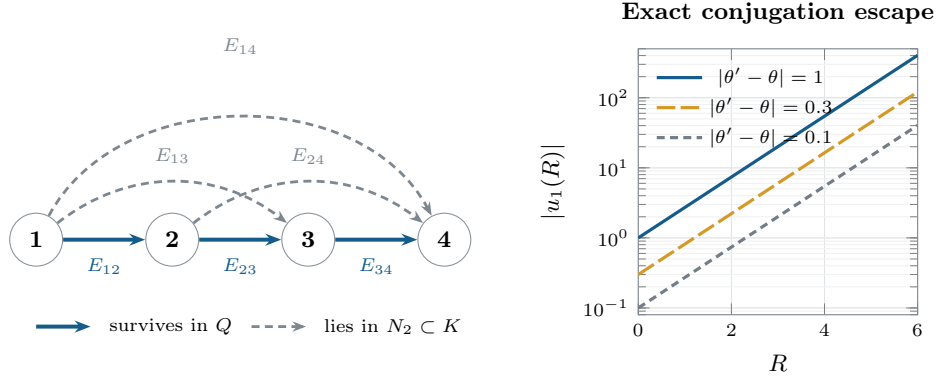

Combining the coefficient decay with the exact conjugation identity now
places this concrete representation within the scope of the abstract storage
theorem.  We therefore obtain the following orbit-frame conclusion.

\begin{corollary}[Concrete orbit frames on $B_4^+$]
\label{cor:B4-orbit-frame}
For every $\xi\in(\R\setminus\{0\})^3$, $\tau\in\R$, and
$0<\varepsilon<1$, there exist a nonzero $\phi\in L^2(\R^3)$ and a
countable set $\Gamma\subset B_4^+$ such that
\[
 (1-\varepsilon)^2\|F\|_2^2
 \leq\sum_{\gamma\in\Gamma}
 |\langle F,\pi_{\xi,\tau}(\gamma)\phi\rangle|^2
 \leq(1+\varepsilon)^2\|F\|_2^2
\]
for every $F\in L^2(\R^3)$.  Moreover, the quotient map is injective on
$\Gamma$, and its image in
$Q\simeq\operatorname{Aff}^+(\R)^3$ is left relatively separated.
\end{corollary}

\begin{proof}
Take $A=U=(\R,+)$, use the homomorphism $p$ displayed above, and set
$h_m=b(m)$.  Proposition~\ref{prop:B4-direct-C0} supplies projective $C_0$
decay, while the computation of $q(b(R))$ preceding
\eqref{eq:B4-exact-escape}, together with that identity, supplies the two
required quotient-escape limits.  The conclusion now follows from
Corollary~\ref{cor:projective-storage}.
\end{proof}

\begin{remark}[Exact scope of the family]
The parameters $\xi_j\neq0$ form eight open sign orthants in
$(N/[N,N])^\wedge\simeq\R^3$, with the diagonal group acting transitively on
each orthant, while varying $\tau$ supplies the central characters of $Z_+$.
We obtain in this way a substantial family of irreducible representations of
the full group $B_4^+$, without asserting a classification of its entire
unitary dual.  Such a classification would also have to account for
representations arising from higher-dimensional irreducible representations
of $N$ and would require additional orbit or Mackey analysis, none of which
is needed for the present application.
\end{remark}

\appendix

\section{Scope of the machine-checked supplement}
\label{app:lean}

With the analytic argument and the triangular case study complete, we now
record precisely which components have been checked in Lean and provide
the filenames of the corresponding source files.
The companion has two scoped layers.  A dependency-free component record,
consisting of seven Lean source files, was checked with Lean~4.32.1.  The
substantive Hilbert-space layer consists of thirteen Lean source files and is
pinned to Lean~4.28.0 and Mathlib~4.28.0.  To keep the present
representation-theoretic scope visible, the module names below identify the
relevant source files rather than an aggregate entry point.  A source-only
directory named \texttt{LeanCertificate} accompanies the arXiv submission as
ancillary material.

\begin{center}
\fcolorbox{OrbitBlue}{OrbitMist}{%
\parbox{0.90\linewidth}{\centering
\textbf{Lean source directory.}\enspace
On arXiv, the complete source-only certificate appears in the ``Ancillary
files'' section of the article's abstract page.  The filenames listed below
identify the corresponding modules inside the \texttt{LeanCertificate}
directory.}}
\end{center}

Each check mark applies only to the result or reusable component stated in
that item; it does not certify a surrounding manuscript theorem unless the
item explicitly says ``end to end.''  The certificate is hash-bound to this
manuscript for identification, but that binding does not enlarge the
formalized boundary.  Thus no check mark should be read as an end-to-end
certification of the entire \TeX{} file.

We first list the result that is certified from its hypotheses through its
conclusion.

\noindent\textbf{Certified end to end.}
\begin{itemize}[label=\(\checkmark\),leftmargin=*]
\item Lemma~\ref{lem:synthesis-perturbation}, in its genuine complex
Hilbert-space form with bounded synthesis maps, adjoints, and operator norms.
Its exact Lean declaration is \texttt{revised\_synthesis\_perturbation},
defined in
\LeanSourceLink{RevisedAnalyticWrapper.lean}.
\end{itemize}

The remaining check marks concern reusable components that enter longer
arguments but do not, by themselves, certify the corresponding manuscript
theorems end to end.

\noindent\textbf{Certified reusable components.}
\begin{itemize}[label=\(\checkmark\),leftmargin=*]
\item The repeated-Hilbert-basis Parseval identity, the square-root past-error
estimate, the rational error budget, and the final conditional frame bounds
are formalized in
\LeanSourceLink{RevisedAnalyticWrapper.lean}.

\item Construction of a bounded synthesis map from square-summable columns,
including its single-column formula and norm bound, is formalized in
\LeanSourceLink{ColumnSynthesisWrapper.lean}.

\item The finite Bessel counting estimate with constant
$4B/\|\phi\|^2$, the four-term complex polarization identity, and the
diagonal-to-cross decay theorem are formalized in
\LeanSourceLink{RevisedPreliminaries.lean}.  The corresponding axiom audit is
recorded in \LeanSourceLink{RevisedPreliminariesAxiomAudit.lean}.

\item Exact affine-ray escape, distinct-direction escape, and a common finite
escape threshold are formalized in
\LeanSourceLink{DirectionalAffineEscape.lean}.  The corresponding axiom audit
is recorded in \LeanSourceLink{DirectionalAffineEscapeAxiomAudit.lean}.
\end{itemize}

To make the boundary explicit, these components contribute to
Lemma~\ref{lem:bessel-relative-separation},
Corollary~\ref{cor:cross-c0}, and the functional-analytic estimates inside
Theorem~\ref{thm:abstract-storage}; they do not furnish an end-to-end
certificate for any of those results.  The weak-operator escape lemma, the
sequential placement and infinite Schur argument, the general
projective-kernel and projective-$C_0$ inputs, the exponential Lie-algebra
escape theorem, the Weyl--Stone--von Neumann branch, the semisimple and
local-field extensions, and the $B_4^+$ case study remain outside the
end-to-end formalized boundary.  Accordingly, Theorems
\ref{thm:abstract-storage}, \ref{thm:main}, \ref{thm:simple-main}, and
\ref{thm:local-field-main}, as well as Corollary~\ref{cor:B4-orbit-frame},
are not marked as Lean-certified.

Finally, the ancillary directory includes a clean-replay script and a
source-hash manifest.  Clean replays of both pinned projects succeeded;
moreover, a strict scan found no proof placeholders or project axioms.  The audited
endpoints use only the standard
Lean/Mathlib axioms \texttt{propext}, \texttt{Classical.choice}, and
\texttt{Quot.sound}.

\section*{Acknowledgments and disclosure of computational assistance}

The author acknowledges substantial assistance from OpenAI's ChatGPT and
Codex in reorganizing the exposition, stress-testing proof steps, checking
constants and operator identities, and preparing \LaTeX{} and limited formal
experiments.  All generated material was treated as fallible draft material.
The author made the final mathematical and editorial decisions and assumes
sole responsibility for every statement, proof, citation, and formulation.


\begin{thebibliography}{99}

\bibitem{ArnalCurrey2020}
D.~Arnal and B.~Currey,
\emph{Representations of solvable Lie groups: Basic theory and examples},
New Mathematical Monographs, Cambridge University Press, Cambridge, 2020.

\bibitem{BaggettKleppner1973}
L.~Baggett and A.~Kleppner,
\emph{Multiplier representations of abelian groups},
J. Funct. Anal. \textbf{14} (1973), no.~3, 299--324.

\bibitem{BekkaLudwig1990}
M.~E.~B. Bekka and J.~Ludwig,
\emph{Complemented $*$-primitive ideals in $L^1$-algebras of exponential Lie
groups and of motion groups},
Math. Z. \textbf{204} (1990), no.~4, 515--526.

\bibitem{BeltitaVanVelthoven2024}
I.~Belti\c{t}\u{a} and J.~T. van Velthoven,
\emph{Symplectic projective orbits of unimodular exponential Lie groups},
Bull. Sci. Math. \textbf{194} (2024), Paper No.~103455.

\bibitem{Ciobotaru2014}
C.~Ciobotaru,
\emph{A unified proof of the Howe--Moore property},
J. Lie Theory \textbf{25} (2015), no.~1, 65--89.

\bibitem{DaubechiesGrossmannMeyer1986}
I.~Daubechies, A.~Grossmann, and Y.~Meyer,
\emph{Painless nonorthogonal expansions},
J. Math. Phys. \textbf{27} (1986), no.~5, 1271--1283.

\bibitem{DuffinSchaeffer1952}
R.~J. Duffin and A.~C. Schaeffer,
\emph{A class of nonharmonic Fourier series},
Trans. Amer. Math. Soc. \textbf{72} (1952), 341--366.

\bibitem{DufloMoore1976}
M.~Duflo and C.~C. Moore,
\emph{On the regular representation of a nonunimodular locally compact group},
J. Funct. Anal. \textbf{21} (1976), no.~2, 209--243.

\bibitem{FeichtingerGrochenig1988}
H.~G. Feichtinger and K.~Gr\"ochenig,
\emph{A unified approach to atomic decompositions via integrable group
representations},
in \emph{Function spaces and applications (Lund, 1986)}, Lecture Notes in
Math., vol.~1302, Springer, Berlin, 1988, pp.~52--73.

\bibitem{FeichtingerGrochenig1989I}
H.~G. Feichtinger and K.~H. Gr\"ochenig,
\emph{Banach spaces related to integrable group representations and their
atomic decompositions. I},
J. Funct. Anal. \textbf{86} (1989), no.~2, 307--340.

\bibitem{FeichtingerGrochenig1989II}
H.~G. Feichtinger and K.~H. Gr\"ochenig,
\emph{Banach spaces related to integrable group representations and their
atomic decompositions. II},
Monatsh. Math. \textbf{108} (1989), no.~2--3, 129--148.

\bibitem{Folland1989}
G.~B. Folland,
\emph{Harmonic analysis in phase space},
Annals of Mathematics Studies, vol.~122, Princeton University Press,
Princeton, NJ, 1989.

\bibitem{FollandAbstract2016}
G.~B. Folland,
\emph{A course in abstract harmonic analysis}, second ed.,
CRC Press, Boca Raton, FL, 2016.

\bibitem{FornasierRauhut2005}
M.~Fornasier and H.~Rauhut,
\emph{Continuous frames, function spaces, and the discretization problem},
J. Fourier Anal. Appl. \textbf{11} (2005), no.~3, 245--287.

\bibitem{FreemanSpeegle2019}
D.~Freeman and D.~Speegle,
\emph{The discretization problem for continuous frames},
Adv. Math. \textbf{345} (2019), 784--813.

\bibitem{FuehrGrochenig2007}
H.~F\"uhr and K.~Gr\"ochenig,
\emph{Sampling theorems on locally compact groups from oscillation estimates},
Math. Z. \textbf{255} (2007), no.~1, 177--194.

\bibitem{GrochenigRottensteiner2018}
K.~Gr\"ochenig and D.~Rottensteiner,
\emph{Orthonormal bases in the orbit of square-integrable representations of
nilpotent Lie groups},
J. Funct. Anal. \textbf{275} (2018), no.~12, 3338--3379.

\bibitem{HoweMoore1979}
R.~E. Howe and C.~C. Moore,
\emph{Asymptotic properties of unitary representations},
J. Funct. Anal. \textbf{32} (1979), no.~1, 72--96.

\bibitem{Husemoller1994}
D.~Husemoller,
\emph{Fibre bundles}, third ed.,
Graduate Texts in Mathematics, vol.~20, Springer-Verlag, New York, 1994.

\bibitem{Kleppner1965}
A.~Kleppner,
\emph{Multipliers on abelian groups},
Math. Ann. \textbf{158} (1965), 11--34.

\bibitem{Knapp2002}
A.~W. Knapp,
\emph{Lie groups beyond an introduction}, second ed.,
Progress in Mathematics, vol.~140, Birkh\"auser Boston, Boston, MA, 2002.

\bibitem{LeptinLudwig1994}
H.~Leptin and J.~Ludwig,
\emph{Unitary representation theory of exponential Lie groups},
De Gruyter Expositions in Mathematics, vol.~18, Walter de Gruyter, Berlin,
1994.

\bibitem{Oussa2018}
V.~Oussa,
\emph{Frames arising from irreducible solvable actions. I},
J. Funct. Anal. \textbf{274} (2018), no.~4, 1202--1254.

\bibitem{Oussa2019}
V.~Oussa,
\emph{Compactly supported bounded frames on Lie groups},
J. Funct. Anal. \textbf{277} (2019), no.~6, 1718--1762.

\bibitem{Oussa2024}
V.~Oussa,
\emph{Orthonormal bases arising from nilpotent actions},
Trans. Amer. Math. Soc. \textbf{377} (2024), no.~2, 1141--1181.

\bibitem{Oussa2026}
V.~Oussa,
\emph{Two open problems on orbit frames for Lie groups},
research preprint, 2026, doi:10.13140/RG.2.2.26162.98246.

\bibitem{Perelomov1972}
A.~M. Perelomov,
\emph{Coherent states for arbitrary Lie group},
Comm. Math. Phys. \textbf{26} (1972), no.~3, 222--236.

\bibitem{Pukanszky1968}
L.~Puk\'anszky,
\emph{On the unitary representations of exponential groups},
J. Funct. Anal. \textbf{2} (1968), no.~1, 73--113.

\bibitem{Simms1970}
D.~J. Simms,
\emph{Topological aspects of the projective unitary group},
Math. Proc. Cambridge Philos. Soc. \textbf{68} (1970), no.~1, 57--60.

\end{thebibliography}
\end{document}